\documentclass[12pt]{article}
\usepackage{amsmath}
\usepackage{amssymb}
\usepackage{accents}

\usepackage{mathtools}
\usepackage{amsthm}
\usepackage[usenames]{color}

\usepackage{graphicx}

\newcommand{\1}[1]{{\mathbf 1}_{\{#1\}}}

\newcommand{\vr}{\varrho}

\newcommand{\eps}{\varepsilon}
\newcommand{\Z}{{\mathbb Z}}

\newcommand{\C}{{\mathbb C}}

\newcommand{\B}{{\mathsf B}}

\newcommand{\V}{{\mathcal V}}

\newcommand{\X}{{\mathcal X}}
\newcommand{\D}{{\mathcal D}}

\newcommand{\R}{{\mathbb R}}
\newcommand{\RR}{{\mathcal R}}

\newcommand{\RI}{\mathop{\mathrm{RI}}}
\newcommand{\BRI}{\mathop{\mathrm{BRI}}}

\newcommand{\s}{{\widehat S}}

\newcommand{\I}{{\mathcal I}}

\newcommand{\QQ}{{\mathcal Q}}
\let\phi=\varphi

\newcommand{\E}{{\mathbb E}}

\newcommand{\ha}{{\hat a}}
\newcommand{\hb}{{\hat b}}
\newcommand{\hc}{{\hat c}}

\newcommand{\tX}{{\widetilde X}}

\newcommand{\tZ}{\widetilde{Z}}

\newcommand{\diam}{{\mathop{\mathrm{diam}}}}

\newcommand{\8}{{\infty}}
\newcommand{\6}{{\partial}}
\newcommand{\nn}{{\nonumber}}

\newcommand{\eqlaw}{\stackrel{\text{\tiny law}}{=}}

\newcommand{\convlaw}{\stackrel{\text{\tiny law}}{\longrightarrow}}

\newcommand{\IP}{{\mathbb P}}
\newcommand{\IE}{{\mathbb E}}
\newcommand{\Cor}{\mathop{\mathrm{Cor}}}

\DeclareMathSymbol{\widehatsym}{\mathord}{largesymbols}{"62}
\newcommand\lowerwidehatsym{%
  \text{\smash{\raisebox{-1.3ex}{%
    $\widehatsym$}}}}
\newcommand\fixwidehat[1]{%
  \mathchoice
    {\accentset{\displaystyle\lowerwidehatsym}{#1}}
    {\accentset{\textstyle\lowerwidehatsym}{#1}}
    {\accentset{\scriptstyle\lowerwidehatsym}{#1}}
    {\accentset{\scriptscriptstyle\lowerwidehatsym}{#1}}
}

\newcommand{\hW}{\widehat{W}}

\newcommand{\hZ}{\widehat{Z}}

\newcommand{\JJ}{{\mathfrak J}}

\newcommand{\capa}{\mathop{\mathrm{cap}}}

\newcommand{\Var}{\mathop{\mathrm{Var}}}

\newcommand{\hm}{\mathop{\mathrm{hm}}\nolimits}
\newcommand{\hhm}{\mathop{\widehat{\mathrm{hm}}}\nolimits}
\newcommand{\dist}{\mathop{\mathrm{dist}}}

\newcommand{\htau}{\widehat{\tau}}

\allowdisplaybreaks

\newtheorem{theo}{Theorem}[section]
\newtheorem{lem}[theo]{Lemma}
\newtheorem{df}[theo]{Definition}
\newtheorem{prop}[theo]{Proposition}
\newtheorem{cor}[theo]{Corollary}
\newtheorem{rem}[theo]{Remark}

\numberwithin{equation}{section}

\title{Two-dimensional Brownian random interlacements}
\author{Francis Comets$^{1}$ \and
 Serguei~Popov$^{2}$}

\begin{document}

\maketitle


{\footnotesize 
\noindent $^{~1}$Universit\'e Paris Diderot. Math\'ematiques, 8 place Aur\'elie Nemours, 75013 Paris, France. \\
 Laboratoire de Probabilit\'es, Statistique et Mod\'elisation (LPSM), UMR 8001 CNRS-UPD-SU.
\\
\noindent e-mail:
\texttt{comets@lpsm.paris}

\noindent $^{~2}$Department of Statistics, Institute of Mathematics,
 Statistics and Scientific Computation, University of Campinas --
UNICAMP, rua S\'ergio Buarque de Holanda 651,
13083--859, Campinas SP, Brazil\\
\noindent e-mail: \texttt{popov@ime.unicamp.br}

}

\begin{abstract}
We introduce the model of two-dimensional continuous random
 interlacements, which is constructed using the Brownian
trajectories \emph{conditioned} on not hitting a fixed set 
(usually, a disk). This model yields the local picture of Wiener sausage on the torus around a late point. As such, it can be seen as a continuous analogue
of discrete two-dimensional random interlacements~\cite{CPV16}.
At the same time, one can view it 
as (restricted) Brownian loops through infinity.
We establish a number of results analogous to these of~\cite{CP16,CPV16},
as well as the results specific to the continuous case.
\\[.3cm]\textbf{Keywords:} Brownian motion, conditioning,
transience, Wiener moustache, logarithmic capacity, Gumbel process.
\\[.3cm]\textbf{AMS 2010 subject classifications:}
Primary: 60J45; Secondary: 60G55, 60J65, 60K35.

\end{abstract}

{\small
\tableofcontents
}

\section{Introduction}
\label{s_intro}
The model of  random interlacements in dimension $d\geq 3$ has been introduced in the discrete setting in~\cite{Szn10}  to describe the local picture of the trace left by the random walk on a large 
torus at a large time. 
It consists in a Poissonian soup of (doubly infinite) 
random walk paths modulo time-shift,  which is a natural and general 
construction~\cite{LeJanSF}. 
It has soon attracted the interest of the community, 
the whole field has now come to a maturity, as can be seen 
in two books~\cite{CT12,DRS14} dedicated to the model.

 In the continuous setting and dimension $d\geq 3$, 
Brownian  random interlacements
bring the similar limit description for the Wiener sausage 
on the torus~\cite{LS15,Szn13}.  
Continuous setting is very appropriate to capture the geometrical properties of the sausage in their full  complexity, 
see~\cite{GdH14}. We also mention the ``general'' 
definition~\cite{R14} of continuous RI,
and an alternative definition of Brownian interlacements
 via Kuznetsov measures~\cite{DD15}.

The model of two-dimensional (discrete) random interlacements was
introduced in~\cite{CPV16}.  (For completeness, we mention here the 
one-dimensional case that has been studied in~\cite{CDarcyP16}.)
The point is that, in two dimensions, even a single trajectory
of a simple random walk is space-filling, so the existence 
 of the model 
has to be clarified together with its  meaning.
To define the process in a meaningful way on the planar lattice, 
one uses the SRW's trajectories \emph{conditioned}
on never hitting the origin, and it gives the local limit of the trace 
of the walk on a large torus at a large time given that the origin has 
not been visited so far. 
We mention also the recent alternative definition~\cite{Rodr17} of the 
two-dimensional  random interlacements in the spirit of thermodynamic limit.
A sequence of finite volume approximation is introduced, consisting in 
random walks killed with intensity depending on the spatial scale
and restricted to paths avoiding the origin. 
Two versions are presented, using
Dirichlet boundary conditions outside a box  or by killing the walk 
at an independent exponential random time.

The interlacement gives the structure of late points and accounts for their clustering studied in~\cite{DPRZ06}. 
They convey information on the cover time, for which the LLN was 
obtained in~\cite{DPRZ04}. 
Let us mention recent results for the cover time in two dimensions in the continuous case: 
(i) computation of the second order correction 
to the typical cover time on the 
torus~\cite{BK14}, 
(ii) tightness of the cover time relative to its median 
for Wiener sausage on the Riemann sphere~\cite{BRZ17}.  
\medskip

In this paper we define the two-dimensional Brownian random interlacements,
implementing the program of~\cite{CPV16} in the continuous setting;
similarly to the discrete case, they are made of \emph{conditioned}
(on not hitting the unit disk) Brownian paths.
Again, similarly to the discrete case, it holds that 
the Brownian random interlacements arise as a limit of the picture
seen from a fixed point on the torus, given that this point
remains uncovered by the Brownian sausage.
In fact, we find that the purity of the concepts and the geometrical 
interest of the interlacement and of its vacant set are uncomparably
stronger here. From a different perspective, we also obtain fine 
properties of Brownian loops through infinity
which shows that the two models are equivalent by inversion in 
the complex plane.

%
%

%
%
%

For our purpose, we introduce the Wiener moustache, which is the doubly 
infinite path used in the soup. The Wiener moustache is obtained by pasting
two independent Brownian motions conditioned to stay outside the unit
disk starting from a random point on the circle.
The conditioned  Brownian motion is itself an interesting object,
it seems to be overlooked in spite of the extensive literature on
the planar Brownian motion, cf.~\cite{LeGall}.
It is defined as a
Doob's $h$-transform, see Chapter X of Doob's book~\cite{Doob84} about conditional Brownian motions; see also~\cite{Doob57} for an earlier approach via transition semigroup.
A modern exposition of methods and difficulties in conditioning
a process by a negligible set is~\cite{RoynetteYor}.
For a  perspective from non equilibrium physics, 
see~\cite[Section 4]{ChetriteTouchette}.

We also study the process of distances of a point to the BRI as 
the level increases. After a non-linear transformation in time 
and space, the process 
has a large-density limit given by a pure jump process with a drift. 
The limit is stationary with the negative of a Gumbel distribution 
for invariant measure, and it is in fact related to models 
for congestion control on Internet (TCP/IP), 
see~\cite{baccelli07, baccelli09}.

\section{Formal definitions and results}
\label{s_defs_results}
In Section~\ref{s_moustache}
we discuss the Brownian motion conditioned on never returning
to the unit disk (this is the analogue of the walk~$\s$ in the discrete
case, cf.~\cite{CP16,CPV16}), 
and define an object called \emph{Wiener moustache},
which will be the main ingredient for constructing the 
Brownian random interlacements in two dimensions. 
In Section~\ref{s_df_BRI}
we formally define the Brownian interlacements, and 
in Section~\ref{s_results} we state our main results.

In the following, we will identify~$\R^2$ and~$\C$ 
via $x=(x_1,x_2)=x_1+i x_2$,
$\|\cdot\|$ will denote the Euclidean norm
in~$\R^2$ or $\Z^2$ as well as the modulus in~$\C$, 
and let $\B(x,r)=\{y: \|x-y\|\leq r\}$ be the 
(closed) disk of radius~$r$ centered in~$x$,
and abbreviate $\B(r):=\B(0,r)$.

\subsection{Brownian motion conditioned on staying outside a disk, 
and Wiener moustache}
\label{s_moustache}

Let $W$ be a standard two-dimensional Brownian motion.
For $A\subset \R^2$ define the stopping time
\begin{equation}
 \label{entrance_W}
 \tau(A) = \inf\{t> 0: W_t\in A\},
\end{equation}
and let us use the shorthand $\tau(r):=\tau\big(\partial \B(r)\big)$.
It is well known that $h(x)=\ln \|x\|$ is a fundamental solution of Laplace equation, 
\begin{equation}
\nn
\frac{1}{2} \Delta h = \pi \delta_0,
\end{equation}
with $\delta_0$ being the Dirac mass at the origin
and 
$\Delta={\partial_{x_1}^2}+{\partial_{x_2}^2}$ 
 the Laplacian.
 
An easy consequence of~$h$ being harmonic away from the origin 
and of the optional stopping theorem is that,
for any $0<a<\|x\|<b<\infty$, 
\begin{equation}
\label{hitting_BM}
 \IP_x[\tau(b)<\tau(a)] = \frac{\ln (\|x\|/a)}{\ln (b/a)}\;.
\end{equation}
Since $h$ is non-negative outside the ball $\B(1)$ and vanishes 
on the boundary, a further consequence of harmonicity is that, under $\IP_x$
for $\|x\| > 1$, 
\begin{equation}\nn
 \1{t <\tau(1)} \frac{h(W_t)}{h(x)} \equiv  \frac{h (W_{t\wedge\tau(1)})}{h(x)}
\end{equation}
is a non-negative martingale with mean 1. Thus,
the formula
\begin{equation}
\label{def:hat_W} 
\IP_x \big[ \hW \in A \big] 
=\E_x \Big(\1{W \in A} \1{t <\tau(1)} \frac{h(W_t)}{h(x)} \Big)\;,
\end{equation}
for all $t >0$, for $A \in {\cal F}_t$ ($\sigma$-field generated by the evaluation maps at times $\leq t$ in $ {\cal C}(\R_+, \R^2 )$),
 defines a continuous process $\hW$ on $[0,\8)$ starting at $x$ and taking  values in $\R^2 \setminus \B(1)$.

The process defined as the Doob's $h$-transform of the standard
Brownian motion by the function $h(x)=\ln\|x\|$,  can be seen as the \emph{Brownian motion conditioned on never
hitting}~$\B(1)$, as it appears in  Lemma~\ref{l_conn_W_hW}  below.
Similarly to~\eqref{entrance_W}, we define
\begin{equation}
 \label{entrance_hW}
 \htau(A) = \inf\big\{t> 0: \hW_t\in A\big\}
\end{equation}
and use the shorthand 
$\htau(r):=\htau\big(\partial \B(r)\big)$.
For a $\R^2$-valued process $X=(X_t, t\geq 0)$ 
we will distinguish its geometric range~$X_{\I}$  
on some time interval $\I$ from its restriction $X_{\big|\I}$,
\begin{equation} 
\label{eq:distinguish}
X_{\I} = \bigcup_{t \in \I} \{X_t\}\;,
\qquad X_{\big \vert \I}: t \in \I \mapsto X_t\;. 
\end{equation} 
\begin{lem}
\label{l_conn_W_hW}
 For all~$R>1$ and all $x\in\R^2$ such that $1<\|x\|<R$ we have
\begin{equation}
\label{eq_conn_W_hW}
\IP_x\Big[W_{\big \vert [0,\tau(R)]}
\in\cdot \mid \tau(R)<\tau(1)\Big] 
  = \IP_x \Big[\hW_{\big \vert [0,\htau(R)]}\in\cdot\,\Big].
\end{equation}
\end{lem}
\begin{proof} 
  From~\eqref{def:hat_W} it follows that for $A \in {\cal F}_{\tau(R)}$, 
\begin{align*}
\IP_x \big[ \hW \in A \big] &=\E_x \Big(\1{W \in A} \1{\tau(R) <\tau(1)} \frac{h(W_{\tau(R)})}{h(x)} \Big)\\
&=\E_x \Big(\1{W \in A} \1{\tau(R) <\tau(1)} \Big) \frac{h(R)}{h(x)} 
\;,
\end{align*}
 which is the desired equality in view of~\eqref{hitting_BM}.
\end{proof}
 From~\eqref{def:hat_W} we derive 
 the transition kernel of $\hW$: for  $\|x\|>1, \|y\|\geq 1$, 
\begin{equation}
\label{df_hat_p}
 {\hat p}(t,x,y) = p_0(t,x,y)\frac{\ln\|y\|}{\ln\|x\|} .
\end{equation}
where~$p_0$ denotes
the transition subprobability density of~$W$ killed on hitting
 the unit disk~$\B(1)$. 
%
%
%
%
%
%
%
Thus,
the semigroup $\fixwidehat{P}_t$ of the process~$\hW$ is given by
\begin{equation*}
\fixwidehat{P}_t f (x) = h(x)^{-1} P^0_t\big(hf\big)(x) ,
\end{equation*}
 for bounded functions $f$ vanishing on $\B(1)$,
where $P^0_t= e^{(t/2)\Delta_0}$ with~$\Delta_0$ being
the Laplacian with Dirichlet boundary conditions
on $\R^2 \setminus \B(1)$. From the above formula we compute 
its generator,
\begin{align} \nn
\fixwidehat{L}f(x) &=  \lim_{t \searrow 0} t^{-1} 
\big[  \fixwidehat{P}_t f (x) - f(x) \big]  \\ \nn
 &=  h(x)^{-1} \lim_{t \searrow 0} t^{-1} 
\big[ P^0_t (hf) (x) - (hf)(x) \big]  \\ 
 &= \frac{1}{2 h(x)}  \Delta \big(hf\big)(x) 
 \label{generatorsym} \\ 
&=
\frac{1}{2} \Delta f + \frac{x}{\|x\|^2\ln \|x\|} \cdot \nabla f ,
 \label{generator}
\end{align}
using $\Delta_0 (hf)=\Delta (hf)$ and
harmonicity of~$h$.  
Thus the diffusion~$\hW$
obeys the stochastic differential equation
\begin{equation}
\label{differential_W}
 d\hW_t = \frac{\hW_t}{\|\hW_t\|^2\ln \|\hW_t\|} dt
    +dW_t.
\end{equation}
Sometimes it will be useful to consider an alternative definition of 
the diffusion~$\hW$ using polar coordinates, 
$\hW_t = (\RR_t\cos \Theta_t, \RR_t\sin \Theta_t)$. With~$W^{(1,2)}$
two independent standard linear Brownian motions, consider the stochastic differential equations
\begin{align}
 d \RR_t &= \Big(\frac{1}{\RR_t\ln \RR_t} + \frac{1}{2\RR_t}\Big)dt 
 + dW_t^{(1)},
   \label{df_Rt} \\ 
\intertext{and}    
\label{df_Theta_t}
d \Theta_t &=  \frac{1}{\RR_t} dW_t^{(2)}
\end{align}
where the diffusion~$\Theta$ takes values on the whole~$\R$;
it is an easy exercise in stochastic calculus to show
that~\eqref{differential_W} is equivalent 
to~\eqref{df_Rt}--\eqref{df_Theta_t}. 

Since the norm of the 2-dimensional Brownian motion~$W$ 
is a BES$^2$ process and~$\hW$ is~$W$ conditioned to
 have norm larger than~$1$ (recall Lemma~\ref{l_conn_W_hW}),
the process $\RR=\|\hW\|$ is itself a BES$^2$ conditioned 
on being in~$(1,\8)$. 
For further use, we give an alternative proof of this fact. 
The BES$^2$ process 
has generator and scale function\footnote{Recall that the scale 
function of a one-dimensional diffusion is a strictly monotonic 
function~$s$ such that, for all $a<x<b$, the probability 
starting at~$x$ to exit interval $[a,b]$ to the right is equal
 to $(s(x)-s(a))/(s(b)-s(a))$.}
 given by
\begin{equation} 
\label{eq:BES2}
L_{{\rm BES}^2}f(r)= \frac{1}{2r} 
\big( rf'(r)\big)'\;,\qquad s(r)= \ln r.
\end{equation}
Following Doob~\cite{Doob57}, the infinitesimal generator of 
BES$^2$ conditioned on being in~$[1, \8)$ is 
\begin{equation}
 \label{eq:BES2cond}
[s \! - \! s(1)]^{-1} L_{{\rm BES}^2}
\big( [s \! - \! s(1)] f \big) 
= \frac{1}{2} \Big(f'' + \Big(\frac{1}{r\ln r} 
+ \frac{1}{2r}\Big) f'\Big) ,
\end{equation}
which coincides with the one of the process~$\RR$.

It is elementary to obtain from~\eqref{df_Rt} that 
\begin{equation} \nn
   d\frac{1}{\ln \RR_t} = \frac{-1}{\RR_t\ln^2 \RR_t} dW_t^{(1)},
\end{equation}
so $(\ln \RR_t)^{-1} = \big(\ln\|\hW_t\|\big)^{-1}$ is a local martingale. 
(Alternatively, this can be seen directly from~\eqref{df_hat_p}.)

We will need to consider  the process~$\hW$ 
starting from $\hW_0=w$ with unit norm. 
Since the definition~\eqref{differential_W} makes sense only when 
starting from $w, \|w\|>1$, we now extend it to the case when $\|w\|=1$.
Consider $X$ a 3-dimensional Bessel process 
BES$^3(x)$, i.e., the norm of a 
3-dimensional Brownian motion starting form a 
point of norm~$x$. It solves the SDE
\begin{equation} \nn
dX_t= \frac{1}{X_t} dt + d B_t , \qquad X_0=x \geq 0, 
\end{equation}
with $B$ a 1-dimensional Brownian motion on some probability
space $(\Omega, {\mathcal A}, P)$. Denoting by ${\mathcal F}_t$ 
the $\sigma$-field generated by $B$ up to time $t$,
we consider the probability measure $Q$ on ${\mathcal F}_\8$
given by ${\frac{d Q}{d P}}\big|_{{\mathcal F}_t}= Z_t, t \geq 0$, 
where
\begin{equation}
\label{eq:Z_t}
Z_t = \exp \Big( \int_0^t \phi (X_s) dB_s - \frac{1}{2}
\int_0^t \phi(X_s)^2 ds \Big),
\end{equation}
with
\begin{equation} \nn
\phi (x) = 
\begin{cases}
 \displaystyle\frac{1}{2(1+x)} 
   + \displaystyle\frac{1}{(1+x) \ln (1+x)} 
- \displaystyle\frac{1}{x},
  & \text{ if }   x >0, \\
1/2, \phantom{\int\limits_a^B}
  & \text{ if }  x=0.  
\end{cases}
\end{equation}
Then, $\phi(X_t)$ is progressively measurable and bounded, 
so~$Z_t$ is a $P$-martingale. By Girsanov theorem, we see that 
\[
 \tilde B_t = B_t - \int_0^t \phi(X_s)\, ds
\]
is a Brownian motion under $Q$. Now, the integral formula
\begin{equation} \nn
X_t = X_0 + \int_0^t  \left( \frac{1}{2(1+X_s)} + 
\frac{1}{(1+X_s) \ln (1+X_s)} \right) ds + {\tilde B}_t ,
\qquad t \geq 0, 
\end{equation}
and uniqueness of the solution of the above SDE
show that the law of~$1+X$ under~$Q$ is the law of~$\RR$ under~$P$, 
provided that $X_0>0$.  

\begin{df}
\label{df_RR}
We define the process~$\RR$ starting from
 $\RR_0=1$ in the following way:
it has the same law as $1+X$
under~$Q$ with $X_0=0$. Similarly, we define the law of the process~$\hW$
starting from $w \in \R^2$ with unit norm 
as the law of $(\RR_t, \Theta_t)_t$ 
with~$\RR$ as above and~$\Theta$ given 
by its law
conditional on~$\RR$ as in \eqref{df_Theta_t}. 
\end{df}

Then, the process $\RR$ (respectively, $\hW$)
 is the limit as $\eps\to 0$ of processes started 
from $1+\eps$ 
(respectively, from $(1+\eps)w $).
  This follows from the identities
\begin{equation} \nn
\E f(\RR_t) = E^Q f(1+X_t) = E^P \big[ Z_t f(1+X_t) \big] ,
\end{equation}
with $Z_t$ from \eqref{eq:Z_t} and $X$ depending 
continuously on its initial condition $X_0\geq 0$.

Let us mention some elementary but useful reversibility and scaling
properties of~$\hW$. 
\begin{prop} 
\label{prop:elemhW}
\begin{itemize}
 \item[(i)] The diffusion~$\hW$ is reversible for the measure 
with density $\ln^2 \|x\|$ with respect to the Lebesgue measure 
on $\R^2 \setminus \B(1)$. 
 \item[(ii)] The diffusion~$\RR$ is reversible for the measure 
$r \ln^2 r \, dr$  on $(1,\8)$. 
 \item[(iii)] Let $1 \leq a \leq \|x\|$. 
Then, the law of~$\hW$ starting from~$x$ conditioned on never 
hitting~$\B(a)$ is equal to the law of the 
process $(a \hW(ta^{-2}); t \geq 0)$ with~$\hW$ starting from~$x/a$.
 \item[(iv)] For $1 \leq a \leq r$, the law of~$\RR$ 
 starting from~$r$ conditioned on staying strictly greater than~$a$
 is equal to that of $(a \RR( a^{-2}t))_{t \geq 0}$ with $\RR_0=r/a$. 
\end{itemize}
\end{prop}
\begin{proof} For smooth $f, g : \R^2 \setminus \B(1) \to \R$, we have
\[
\hat{\varepsilon}(f,g):=
- \int_{ \R^2 \setminus \B(1) } f(x) 
\fixwidehat{L}g(x) h^2(x) \, dx \stackrel{}{=} \frac{1}{2} 
\int_{ \R^2 \setminus \B(1) }  {\nabla(hf)} \cdot \nabla (hg)\, dx,
\]
by \eqref{generatorsym} and integration by parts. Since this
is a symmetric function of $f,g$, claim~(i) follows.
For~(ii), using~\eqref{df_Rt}, we write 
the generator of~$\RR$ as 
\begin{equation}\nn
L_{\RR}f(r)= \frac{1}{2r \ln^2 r} \big( f'(r) r \ln^2r \big)' ,
\end{equation}
for smooth $f: (1, \8) \to \R$,
and conclude by a similar computation.

For a Brownian motion~$W$ starting from~$x/a$, 
the scaled process
 $W^{(a)}(t)=aW(t/a^2)$ is a Brownian motion
starting from~$x$. Since~$\hW$ is the Brownian motion conditioned
 on never hitting $\B(1)$,
the process $(a \hW(ta^{-2}); t \geq 0)$ 
 has the law of~$W$ conditioned on never hitting $\B(a)$.
In turn, the latter has the same law as~$\hW$ starting from~$x$
 conditioned on never hitting~$\B(a)$. 
 We have proved~(iii). 
Claim~(iv) follows directly from~(iii) 
and the definition of~$\RR$ as the norm of~$\hW$.
\end{proof}

Since $\big(\!\ln\|\hW_t\|\big)^{-1}$
is a local martingale, the optional stopping theorem implies
that for any $1<a<\|x\|<b<\infty$
\begin{equation}
\label{hitting_condBM}
 \IP_x[\htau(b)<\htau(a)] = \frac{(\ln a)^{-1}-(\ln \|x\|)^{-1}}
 {(\ln a)^{-1}-(\ln b)^{-1}}
   =\frac{\ln(\|x\|/a)  \times \ln b}{\ln (b/a) \times\ln \|x\|}.
\end{equation}
Sending $b$ to infinity in~\eqref{hitting_condBM}
we also see that for $1\leq a\leq \|x\|$
\begin{equation}
\label{escape_condBM}
 \IP_x[\htau(a)=\infty] = 1-\frac{\ln a}{\ln \|x\|}.
\end{equation}
%
%

Now, we introduce one of the main objects of this paper:
\begin{df}
\label{df_moustache}
 Let~$U$ be a random variable with uniform
 distribution on $[0,2\pi]$, and let
 $(\RR^{(1,2)},\Theta^{(1,2)})$
be two independent copies of the two-dimensional 
diffusion on $\R_+\times \R$
defined by \eqref{df_Rt}--\eqref{df_Theta_t}, 
with common initial point~$(1,U)$.
Then, the {\bf Wiener moustache}~$\eta$ is defined as
 the union of ranges of the two trajectories, i.e., 
\[
\eta = \big\{(r,\theta): \text{ there exist }k\in\{1,2\}, t\geq 0
 \text{ such that } \RR^{(k)}_t=r, \Theta^{(k)}_t=\theta\big\}. 
\]
When there is no risk of confusion, 
we also call the Wiener moustache
the image of the above object under 
the map $(r,\theta)\mapsto (r\cos\theta, r\sin \theta)$
(see below).
\end{df}
We stress that we view the trajectory of a process 
as a geometric subset of~$\R^2$, forgetting its time parametrization.

Informally, the Wiener moustache is just the union of
 two independent Brownian trajectories
started from a random point on the boundary of the unit disk, conditioned 
on never re-entering this disk, see Figure~\ref{f_moustache}.
\begin{figure}
\begin{center}
\includegraphics{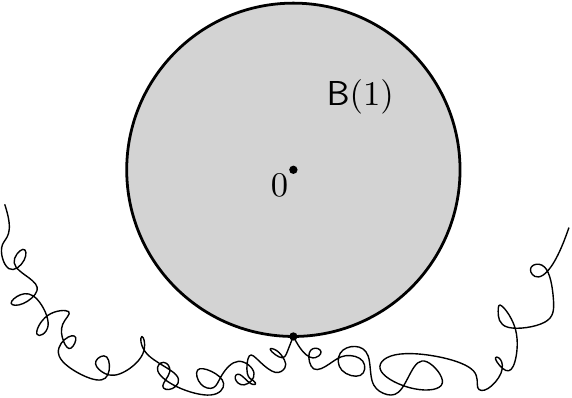}
\caption{An instance of Wiener moustache (for aestetical reasons,
the starting point was randomly chosen to be $x=(0,-1)$).}
\label{f_moustache}
\end{center}
\end{figure}
One can also represent the Wiener moustache as one doubly-infinite
trajectory $(\hW_t, t\in \R)$, where $\|\hW_0\|=1$, $\|\hW_t\|>1$
for all $t\neq 0$.

\begin{rem}
\label{rem_hW_b}
The Brownian motion $\hW^{b}$ conditioned not to enter in the ball~$\B(b)$
 of radius $b>0$ can be defined similarly to above. 
Proposition~\ref{prop:elemhW} (iii) and Lemma~\ref{l_conn_W_hW}
imply that, in law, 
\[
\hW^{b}(t) = b \hW(b^{-2}t)\;,\quad t \geq 0.
\]
Therefore, the set of visited points and the connected component 
of~$0$ in its complement are simply the $b$-homothetics 
of the Wiener moustache.
\end{rem}

\subsection{Two-dimensional Brownian random interlacements}
\label{s_df_BRI}
Now, we are able to define the model of Brownian random 
interlacements in the plane. We prefer not to imitate the corresponding
construction of~\cite{CPV16} of discrete random interlacements
which  uses a general construction of~\cite{T09}
of random interlacements on transient weighted graphs;
instead of defining a consistent family of probability 
measures on closed subsets of bounded regions, we rather 
give an 
``infinite volume description'' of the model.
\begin{df}
\label{df_BRI}
 Let~$\alpha>0$ and consider a Poisson point process
$(\rho_k^\alpha, k\in\Z)$ on~$\R_+$ 
with intensity~$r(\rho)=\frac{2\alpha}{\rho}, \rho\in\R_+$.
Let $(\eta_k, k\in\Z)$ be an independent sequence of i.i.d.\ Wiener moustaches.
Fix also $b\geq 0$.
Then, the model of Brownian Random Interlacements (BRI) on level~$\alpha$
truncated at~$b$ is defined as the following subset of $\R^d \setminus \B(1)$
(see Figure~\ref{f_BRI_def}):
\begin{equation}
\label{eq_BRI}
\BRI(\alpha;b) = \bigcup_{k:\rho_k^\alpha \geq b}\rho_k^\alpha \eta_k\,.
\end{equation}
\end{df}
Let us abbreviate $\BRI(\alpha):=\BRI(\alpha;1)$.
As shown on Figure~\ref{f_BRI_def}
on the left, the Poisson process with 
rate~$r(\rho)=\frac{2\alpha}{\rho}$
can be obtained from a two-dimensional Poisson point process
with rate~$1$ in the first quadrant, 
by projecting onto the horizontal axis those points which lie below
$r(\rho)$. Since the area under~$r(\rho)$
is infinite in the neighborhoods of~$0$ and~$\infty$, there 
is a.s.\ an infinite number of points of the Poisson
process in both $(0,\eps)$ and~$(M,\infty)$ for all 
positive~$\eps$ and~$M$. 

An important observation is that the above Poisson process
is the image of a homogeneous 
 Poisson process of rate~$1$ in~$\R$ under the map 
$x\mapsto e^{x/2\alpha}$
(or, equivalently, the image of a homogeneous 
 Poisson process of rate~$2\alpha$ 
 under the map 
$x\mapsto e^{x}$); this is a straightforward 
consequence of the Mapping Theorem
for Poisson processes (see e.g.\ Section~2.3 of~\cite{K93}).
In particular, we may write
\begin{equation}
\label{rho_exponential}
 \rho_k^\alpha = \exp\Big(\frac{Y_1+\cdots+Y_k}{2\alpha}\Big),
\end{equation}
where $Y_1,\ldots,Y_k$ are i.i.d.\ Exponential(1) 
random variables.

\begin{rem}
\label{rem_simultaneous} 
 From the above, it follows 
that we can construct
$\BRI(\alpha;0)$ -- and hence also $\BRI(\alpha;b)$ for all~$b>0$ --
\emph{simultaneously} for all~$\alpha>0$
in the following way (as shown on the left side
of Figure~\ref{f_BRI_def}): consider a 
\emph{two-dimensional} Poisson point process
of rate~$1$ in~$\R^2_+$, and then take the abscissa of the points 
below the graph of $r(\rho) = \frac{2\alpha}{\rho}$
to be the distances to the origin of the corresponding
Wiener's moustaches. In view of the previous observation,
an equivalent way to do this is to consider a Poisson point process
 of rate~$1$ in~$\R\times \R_+$, take the first coordinates of points
with second coordinate at most~$2\alpha$, and exponentiate.
\end{rem}

Observe also that, by construction, for all positive 
$\alpha, \beta, b$ it holds that
\begin{equation}
\label{superp_BRI}
 \BRI(\alpha;b)\oplus \BRI(\beta;b) \eqlaw \BRI(\alpha+\beta;b),
\end{equation}
where $\oplus$ means superposition of independent copies.

\begin{figure}
\begin{center}
\includegraphics[width=\textwidth]{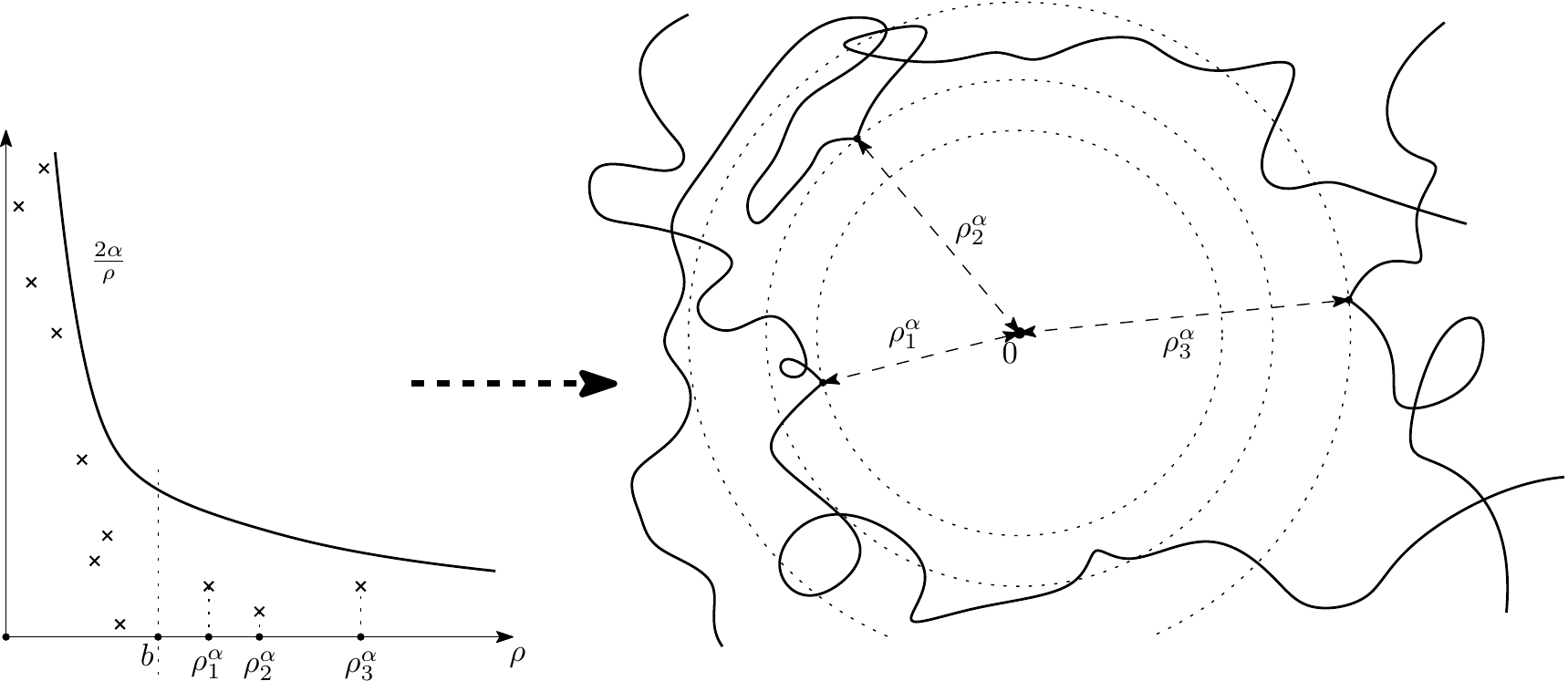}
\caption{On the definition of $\BRI(\alpha;b)$}
\label{f_BRI_def}
\end{center}
\end{figure}

At this point it is worth mentioning that the (discrete) 
random interlacements may be regarded as Markovian loops ``passing
through infinity'', see e.g.\ Section~4.5 of~\cite{Szn12}.
In the continuous case, we note that the object we just constructed can be viewed
as ``Brownian loops through infinity''. More precisely, we have a simple relation to the Brownian loop measure defined in~\cite{LW04} (see also~\cite{LSW03}) and studied in~\cite{W12}:
\begin{theo}
\label{p_Lawler_werner}
 Consider a Poisson process of loops rooted in the origin
with the intensity measure $2\pi\alpha\mu(0,0)$, 
with $\mu(\cdot,\cdot)$ as defined in Section~3.1.1
of~\cite{LW04}.  Then, the inversion (i.e., the image
under the map $z\mapsto 1/z$) of this family of loops
is~$\BRI(\alpha;0)$. The inversion of the loop process
restricted on~$\B(1)$ is~$\BRI(\alpha)$. 
\end{theo}
\begin{proof}
This readily follows from Theorem~1 of~\cite{PY96}
and the invariance of Brownian trajectories under conformal
mappings. 
\end{proof}

\begin{rem}
\label{rem_arb_domain} 
Analogously to~\eqref{df_hat_p}--\eqref{generator} one can 
also define a diffusion~$\hW^{(L)}$ avoiding a 
compact set~$L\subset \C$ 
such that $\C\setminus L$ is simply connected on the Riemann sphere.
Observe that, by the Riemann mapping theorem,
there exists a unique conformal map~$\phi$
that sends the exterior of~$\B(1)$ to
the exterior of~$L$ and also satisfies the conditions
 $\phi(\infty)=\infty$, $\phi'(\infty)>0$.
We then define $\BRI(\alpha;L)$ as $\phi(\BRI(\alpha))$,
see Figure~\ref{f_conformal}.
\end{rem}
\begin{figure}
\begin{center}
\includegraphics{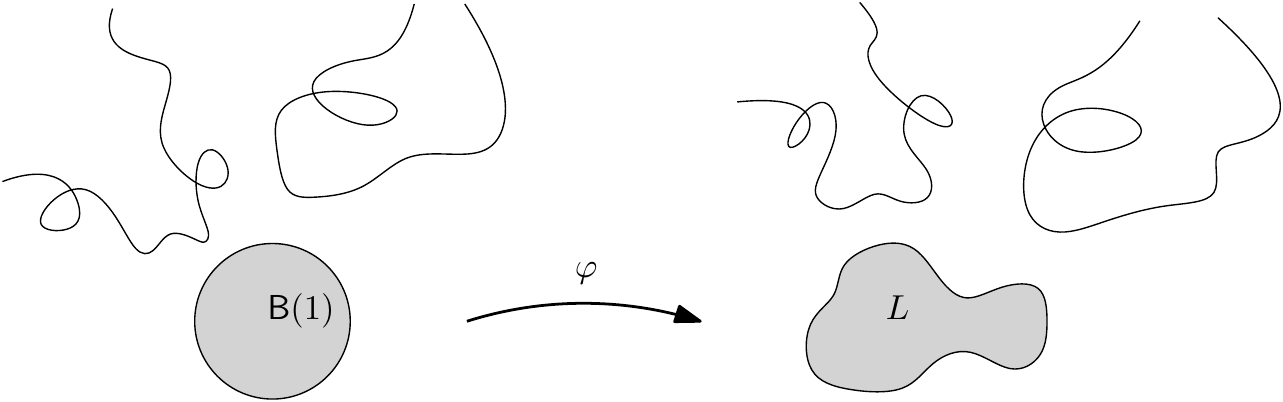}
\caption{On the definition of random interlacements avoiding
the domain~$L$}
\label{f_conformal}
\end{center}
\end{figure}

Next, we need also to introduce  the notion of capacity in the plane. Let~$A$ be a compact subset
of~$\R^2$
such that $\B(1)\subset A$. Let~$\hm_A$ be the \emph{harmonic measure} (from infinity) on~$A$, that is, the entrance law in $A$ for the Brownian motion starting from infinity, cf.\ e.g.\ Theorem~3.46 of~\cite{MP10}. 
We define the 
capacity of~$A$ as
\begin{equation}
\label{df_Br_cap}
 \capa(A) = \frac{2}{\pi}\int_A \ln\|y\|\, d\hm_A(y) .
\end{equation}
We stress that there exist other definitions
of capacity; a common one, called  the \emph{logarithmic capacity}  in Chapter~5 of~\cite{Ransford95}, is given by the
exponential of~\eqref{df_Br_cap} without the constant~$\frac{2}{\pi}$
in front.  However,  in this paper, we prefer the above definition~\eqref{df_Br_cap}
which matches
the corresponding ``discrete'' capacity for random walks,
cf.\ e.g.\ Chapter~6 of~\cite{LL10}.
Note that~\eqref{df_Br_cap} immediately implies that 
 $\capa(A) \geq 0$ for $A \supset \B(1)$,  and 
\begin{equation}
\label{capa_disk}
\capa\!\big(\B(r)\big) = \frac{2}{\pi}\ln r
\end{equation}
for any radius $r$. 
Next, we need to define the harmonic measure~$\hhm_A$ 
for the (transient)  
conditioned diffusion~$\hW$. For a compact, 
non polar set~$A\subset \R^2$ (see, e.g. ~\cite[p.234]{MP10}),
with 
$A\subset \R^2\setminus \mathring{\B}(1)$
and $M\subset\partial A$, let 
\begin{equation} \nn
\hhm_A(M) = \lim_{\|x\|\to \infty} \IP_x\big[\hW_{\htau(A)}\in M
 \mid \htau(A)<\infty \big] ,
\end{equation}
where the existence of the limit follows e.g.\ from Lemma~\ref{l_hat_hm}
below. Observe that for any~$A$ as above it holds 
that $\hhm_A\big(\B(1)\cap\partial A\big)=0$.

Now, we show that we have indeed
defined the ``right'' object in \eqref{eq_BRI}:
\begin{prop}
\label{p_indeed_BRI}
%
\begin{itemize}
 \item[(i)]
 for any compact 
 $A\subset \R^2$ such that $\B(1)\subset A$, we have
\begin{equation}
\label{eq_vacant_Bro}
 \IP\big[A\cap\BRI(\alpha)=\emptyset\big]
    =  \exp\big(-\pi\alpha \capa(A)\big).
\end{equation}
Equivalently, for  $A$ compact, $
 \IP\big[A\cap\BRI(\alpha)=\emptyset\big]
    =  \exp\big(-\pi\alpha \capa(A \cup \B(1))\big).
$
 \item[(ii)]
the 
 trace of $\BRI(\alpha)$ on \emph{any} compact 
set~$A$ such that $\B(1)\subset A$ can be obtained
using the following procedure:
\begin{itemize}
 \item take a Poisson($\pi\alpha\capa(A)$) number of particles;
 \item place these particles on the boundary of~$A$
 independently, with distribution~$\hhm_A$;
 \item let the particles do independent $\hW$-diffusions
 (since~$\hW$ is transient, each walk only leaves a 
 nowhere dense 
 trace
 on~$A$).
\end{itemize}
\end{itemize}
\end{prop}
We postpone the proof of this result to Section~\ref{s_proofs}.
Let us stress that~\eqref{eq_vacant_Bro} is \emph{the characteristic
property of random interlacements}, compare to~(2) of~\cite{CPV16}.
As explained just after Definition~2.1 of~\cite{CPV16}, the 
factor~$\pi$ is there just for convenience, since it makes 
the formulas 
cleaner.
Also, the construction in~(ii) of $\BRI(\alpha)$ on~$A$
agrees with the corresponding discrete ``constructive description''
of~\cite{CPV16}, and is presented in larger details just after Definition~2.1 there.

Next, for a generic function $g:\C\mapsto\C$, 
we denote by $g\big(\!\BRI(\alpha;b)\big)$
the image of $\BRI(\alpha;b)$ under the map~$g$.
We will also need from~$g$ that the image of a
Wiener moustache is itself a Wiener moustache.
Thus, one needs to give a special treatment to power functions
of the form~$g(z)=z^\lambda$ for a noninteger~$\lambda$.
We will use the polar representation, so that 
the complex-valued  power function now has a natural definition,
\begin{equation}
\label{df_power}
 \text{for }r>0 \text{ and } \theta\in\R, \quad (r,\theta)^\lambda 
 := (r^\lambda, \lambda\theta);
\end{equation}
for arbitrary $\lambda>0$.
However this transformation does not preserve the law of Wiener moustache. Indeed, 
the angular coordinates of the initial points 
of the moustaches to be uniform in the interval~$[0,2\pi)$.
If we then apply the power map with, say, $\lambda=1/2$, then
all the initial points will have their angular coordinates
in~$[0,\pi)$,  breaking the isotropy. 

In this concrete case this can be repaired by choosing
the angular coordinates of the initial points 
 in the interval~$[0,4\pi)$, but what to do for other
(possibly irrational) values of~$\lambda$?

Here, we present a construction that works for all $\lambda>0$
simultaneously, but uses an additional sequence
$\xi=(\xi_1,\xi_2,\xi_3,\ldots )$ of i.i.d.\
Uniform$[0,2\pi)$ random variables. 
Let~$\Theta_k$ be the angular coordinate of the initial 
point of~$\eta_k$, the $k$th moustache.
For all $\lambda>0$, set 
\begin{align*}
 \Theta_k^{(\lambda)} &= \lambda(\Theta_k - \xi_k)+\xi_k \mod 2\pi\\
   &= \lambda \Theta_k + (1-\lambda)\xi_k \mod 2\pi;
\end{align*}
 observe that $\Theta_k^{(\lambda)}$ is also Uniform$[0,2\pi)$.
Also, $\Theta_k\mapsto \Theta_k^{(\lambda)}$ seen 
as a function on $\R/2\pi \Z$ varies continuously 
in~$\lambda$. 
The extra randomization carried by the sequence~$\xi_k$ 
defines a transformation on the law of $\BRI(\alpha;b)$, where the scaling of the angle is performed with respect to a reference angle $\xi_k$:
if $(r,\theta)$ belongs to the $k$-th moustache, its image by the 
 power function will be 
$(r^\lambda,  \lambda \theta + (1-\lambda)\xi_k)$.

\subsection{Main results}
\label{s_results}
First, let us define the \emph{vacant set} --
set of points of the plane which do not 
belong to trajectories of $\BRI(\alpha;b)$
\begin{equation} \nn
\V^{\alpha;b} = \R^2\setminus \BRI(\alpha;b).
\end{equation}
For $s>0$ let~$\D_s(A)$ be the $s$-interior of $A\subset\R^2$:
\[
  \D_s(A) = \{x:\B(x,s)\subset A\}.
\]
We are also interested in the sets of the form
$\D_s(\V^{\alpha;b})$, the sets of points that are at distance
larger than~$s$ to $\BRI(\alpha;b)$.
Let us also abbreviate $\V^{\alpha}:=\V^{\alpha;1}$.

To formulate our results, we need to define the 
logarithmic potential $\ell_x$ generated by the entrance law in the unit disk of Brownian motion starting at $x$:
for $x\notin\B(1)$,
\begin{equation}
\label{df_ell_y}
  \ell_x = \int_{\partial \B(1)}
\ln\|x-z\| \, H(x,dz)
= \frac{\|x\|^2-1}{2\pi}\int_{\partial \B(1)}
   \frac{\ln\|x-z\|}{\|x-z\|^2}\, dz,
\end{equation}
where $H(x,\cdot)$ is the entrance measure of the Brownian motion  starting from~$x$ into~$\B(1)$, given in~\eqref{Poisson_kernel} below,
also known as the Poisson kernel.

First, we list some properties of Brownian
interlacements 
(Theorems~\ref{t_basic_prop}, \ref{t_basic_size}, \ref{t_torus})
that are ``analogous'' to those of discrete
two-dimensional random interlacements obtained in~\cite{CP16,CPV16}. 
Then, we state the results which are specific for the Brownian
random interlacements (Theorems~\ref{t_BRI_image}, \ref{t_freedisk},
\ref{t_Phi_process}).
We do this for $\BRI(\alpha)$ only, since
$\BRI(\alpha;b)$ can be obtained from $\BRI(\alpha)$
by a linear rescaling (this is easy to see directly, but also observe
that it is a consequence of Theorem~\ref{t_BRI_image} with $\lambda=1$).
\begin{theo}
\label{t_basic_prop}
\begin{itemize}
 \item[(i)] For any $\alpha>0$, $x\in\R^2\setminus \B(1)$
 and for any compact  set
 $A\subset \R^2$, it holds that 
\begin{equation}
\label{properties_BRI_i}
 \IP[A\subset\V^\alpha \mid \B(x,1)\subset \V^\alpha]=
  \IP[-A+x\subset\V^\alpha \mid \B(x,1)\subset \V^\alpha].
\end{equation}
More generally, for all  $\alpha>0$, $x\in\R^2 \setminus \B(1)$,
 $A\subset \R^2$, and any
  isometry~$M$ exchanging $0$ and $x$, we have 
\begin{equation}
\label{properties_BRI_i'}
 \IP[A\subset\V^\alpha \mid \B(x,1)\subset \V^\alpha]=
  \IP[MA \subset\V^\alpha \mid \B(x,1) \subset \V^\alpha].
\end{equation}
We call this property \emph{the conditional translation invariance}.
 \item[(ii)] We have, for $x\notin\B(1)$
\begin{equation}
\label{properties_BRI_ii}
\IP[x\in \D_1(\V^\alpha)] \equiv
\IP[\B(x,1)\subset\V^\alpha]=
\|x\|^{-\alpha}\big(1+O\big((\|x\|\ln\|x\|)^{-1}\big)\big).
\end{equation}
More generally, if $\|x\|>s+1$
\begin{align}
\IP[x  \in  \D_s(\V^\alpha)]
\label{properties_BRI_ii'}
&\equiv 
\IP[\B(x,s)\subset \V^\alpha]\\
& =  \exp \! \Big( \! -2\alpha
\frac{\ln^2\|x\|}
{\ln\|x\| \! + \! \ell_x  \! -  \! \ln s}
\Big(1 \! + \! O\big(\textstyle
\frac{\ln\theta_{x,s}}{\theta_{x,s}}\big(
\frac{1}{\ln(\|x\| +  1)} \! + \! \frac{1}{\ln\theta_{x,s}
 + \ln\|x\|} \! \big) \! \big) \! \Big) \! \Big),
\nn
\end{align}
where $\theta_{x,s}=\frac{\|x\|-1}{s}$.
 \item[(iii)] For compact 
 set~$A$ such that $\B(1)\! \subset \!A\subset \!\B(r)$
and $x\!\in\!\R^2$ such that $\|x\|\!\geq \!2r$,
\begin{equation}
\label{properties_BRI_iii}
 \IP[A\!\subset\!\V^\alpha \mid \B(x,1)\!\subset\! \V^\alpha]= 
   \exp\Bigg(\!-\frac{\pi\alpha}{4}\capa(A)
\frac{1+O\big(\frac{r\ln r }{\|x\|}\big)}
{1\!-\!\frac{\pi\capa(A)}{4\ln\|x\|}\!
+\!O\big(\frac{r\ln r}{\|x\|\ln\|x\|}\big)}
\Bigg)
\end{equation}
 \item[(iv)] Let $x,y\notin \B(1)$.
As $s:= \|x\| \to \infty$, $\ln \|y\|  \sim \ln s$ and
$\ln \|x-y\|\sim \beta \ln s$ with some $\beta\in [0,1]$, we have
\begin{equation*}
 \IP\big[\B(x,1)\cup \B(y,1) \subset \V^\alpha\big]
= s^{-\frac{4\alpha}{4-\beta}+o(1)},
\end{equation*}
and 
polynomially decaying correlations 
\begin{equation}
\label{eq:cor}
\Cor\big(
 {\{\B(x,1)\subset \V^\alpha\}}, 
 {\{\B(y,1)\subset \V^\alpha\}}\big)
= s^{-\frac{\alpha \beta}{4-\beta}+o(1)},
\end{equation}
where $\Cor(A,B) = \frac{ {\rm Cov}({\mathbf 1}_A, {\mathbf 1}_ B)}
{[ \Var({\mathbf 1}_A) \Var(  {\mathbf 1}_ B)]^{1/2}}\in [-1,1]$.
\end{itemize}
\end{theo}
\begin{rem}
When $x\in\partial\B(1)$,
one can also write both explicit and asymptotic as $r\to 0$
formulas for $\IP[x\in\D_r(\V^\alpha)]$ using Lemma~\ref{l_cap_blister}
below.
Also, the results presented in~(iv) above are important because they 
permit us to ``quantify'' the dependence between what happens in
different (distant) places. Such quantitative estimates play an
important role in many renormalization-type arguments, which are frequently
useful in the context of random interlacements.
\end{rem}

  From the definition, it is clear that
$\BRI(\alpha)$ model is not translation invariant
(note also~\eqref{properties_BRI_ii}). 
Therefore, in {\em (iv)} we emphasize estimate~\eqref{eq:cor} 
using the correlation in order to measure the spatial dependence 
because this is a normalized quantity. 

Then, we obtain a few results about the size of the
interior of the vacant set.
\begin{theo}
\label{t_basic_size}
Fix an arbitrary $s>0$.
\begin{itemize}
 \item[(i)] We have, as $r\to\infty$
\[
 \E \big(\vert \D_s(\V^\alpha)\cap \B(r)\vert \big) \sim
  \begin{cases}
   \frac{2\pi}{2-\alpha} \times r^{2-\alpha}, 
& \text{ for }\alpha < 2,\\
  2\pi \times \ln r, & \text{ for }\alpha = 2,\\
   \frac{2\pi}{\alpha-2} , & \text{ for }\alpha > 2.
 \end{cases}
\]
 \item[(ii)] For all $\alpha>1$,  the set $\D_s(\V^\alpha)$ is
  a.s.\ bounded. Moreover,  we have
  $\IP\big[\D_s(\V^\alpha) \subset \B(1-s+\delta) \big]>0$
 for all $\delta>0$, 
 and $\IP\big[\D_s(\V^\alpha) \subset \B(1-s+\delta) \big]\to 1$ 
 as $\alpha\to\infty$.
 \item[(iii)] For all $\alpha \in (0,1]$, the set~$\D_s(\V^\alpha)$
is a.s.\ unbounded.  Moreover, for $\alpha \in (0,1)$ it holds that
\begin{equation}
\label{eq_emptydisk}
 \IP\big[\D_s(\V^\alpha)\cap \big(\B(r)
 \setminus \B(r/2)\big)=\emptyset\big]
 \leq r^{-2(1-\sqrt{\alpha})^2+o(1)}.
\end{equation}
\end{itemize}
\end{theo}
Remarkably, the above results do not depend on the 
value of~$s$. This is due to the fact that
(recall~\eqref{properties_BRI_ii'}), for large~$x$,
\[
 \IP[x\in \D_s(\V^\alpha)] \approx 
\|x\|^{-\frac{\alpha}{1-\frac{\ln s}{2\ln x}}},
\]
so the exponent approaches~$\alpha$ as $x\to \infty$ for any fixed~$s$.
Notice, however, that for very small or very large
values of~$s$ this convergence can be quite slow. 

Now, we state our results for the Brownian motion on the torus. 
Let~$(X_t, t\geq 0)$ be the Brownian motion on~$\R^2_n=\R^2/n\Z^2$
with~$X_0$ chosen uniformly at random\footnote{The reader may 
wonder at this point why we consider a torus 
of linear size~$n$. By scaling 
our results can be formulated on the torus of unit size, 
replacing the Wiener sausage's radius~$1$ by $\eps=1/n$. 
The reason for our choice 
is that in this paper
we study $\BRI(\alpha, b)$ 
with a fixed radius $b=1$, which corresponds to the former case.}.
Define
\begin{equation} \nn
\X_t^{(n)}=\{X_s, s\leq t\} \subset\R^2_n
\end{equation}
to be the set of 
points hit by the Brownian trajectory until time~$t$. 
The Wiener sausage at time $t$ is the set of points on the 
torus at distance less than or equal to 1 from the set $\X_t^{(n)}$.
The cover time is the time when the Wiener sausage covers 
the whole torus.
Denote by $\Upsilon_n:\R^2 \to \R^2_n$ 
the natural projection modulo~$n$:
$\Upsilon_n(x,y)=(x \mod n, y \mod n)$. 
Then, if~$W_0$ were chosen uniformly at random 
on any fixed $n\times n$ square with sides parallel to the axes,
 we can write $X_t=\Upsilon_n (W_t)$.
Similarly,
$\B(y,r)\subset \R^2_n$ is defined by $\B(y,r)=\Upsilon_n \B(z,r)$,
where $z\in\R^2$ is such that $\Upsilon_n z = y$.  
Define also
\[
t_\alpha:=\frac{2\alpha}{\pi}n^2\ln^2 n;
\]
it was proved in the seminal paper \cite{DPRZ04} that 
 $\alpha=1$ corresponds to the leading-order term of the 
expected cover time of the torus, see also \cite{BK14} for the next leading term and \cite{Abe17} in the discrete case.
In the following
theorem, we prove that, given that the unit ball
 is unvisited by the Brownian motion, the law
of the uncovered set around~$0$ at time~$t_\alpha$
is close to that of $\BRI(\alpha)$:
\begin{theo}
\label{t_torus} 
Let $\alpha>0$ and~$A$ be a compact subset of~$\R^2$ 
such that $\B(1)\subset A$. Then, we have 
\begin{equation}
\label{eq_torus}
 \lim_{n\to\infty}\IP\big[\Upsilon_n A \cap 
   \X_{t_\alpha}^{(n)}=\emptyset \mid \B(1)\cap 
   \X_{t_\alpha}^{(n)}=\emptyset\big]
  = \exp\big(-\pi\alpha\capa(A)\big).
\end{equation}
\end{theo}

In fact, Theorems~\ref{t_basic_prop}, \ref{t_basic_size},
and~\ref{t_torus}
can be seen as the continuous counterparts of 
Theorems~2.3, 2.5, and~2.6 of~\cite{CPV16}
and also Theorem~1.2 of~\cite{CP16}
 (for the critical case $\alpha = 1$). 

 From this point on, we discuss some facts specific to the 
continuous-time case (i.e., the \emph{Brownian}
random interlacements). 
We first describe the
scaling properties of two-dimensional Brownian interlacements:
\begin{theo}
\label{t_BRI_image}
For any positive $c$ and~$\lambda$,
  it holds that  
\[
c\times
\BRI(\alpha;b)
 \eqlaw \BRI(\alpha;cb)\quad 
\text{ and } \quad
\big(\!\BRI(\alpha;b)\big)^\lambda  \eqlaw \BRI(\alpha/\lambda;b^\lambda)\;.
\] 
\end{theo}
In a more compact form, the claim is 
$c\times\big(\!\BRI(\alpha;b)\big)^\lambda  
\eqlaw \BRI(\alpha/\lambda;cb^\lambda)$.
\medskip

Next, we discuss some fine properties of two-dimensional Brownian random interlacements as a process indexed by $\alpha$.
We emphasize that the coupling between the different $\BRI$'s as $\alpha$ varies becomes essential in the forthcoming considerations. We recall the definition of $\BRI$  
from Remark~\ref{rem_simultaneous}.

 For 
 $x \in \R^2$  let 
\[
\Phi_x(\alpha)= {\rm dist}\big(x, \BRI(\alpha)\big)
\]
 be the Euclidean distance from~$x$ 
to the closest trajectory in the interlacements. 
 Since $\Phi_0(\alpha) = \rho_1^\alpha$,
 by~\eqref{capa_disk} and~\eqref{eq_vacant_Bro} 
we see that for $s \geq 1$,
\[
\IP[\Phi_0(\alpha) >s] = \IP[\B(s)\subset\V^\alpha]
 = s^{-2\alpha}\;,
 \]
 that is, for all $\alpha>0$
\begin{equation}
\label{Phi_0_exact}
 2\alpha\ln\Phi_0(\alpha) \eqlaw 
 \text{Exp(1) random variable. } 
\end{equation}
It is interesting to note that, for all $x$,  $(\Phi_x(\alpha),\alpha>0)$
is an homogeneous  Markov process:
\begin{theo}
\label{t_Phi_process} 
 The process $(\Phi_x(\alpha),\alpha>0)$ is a nonincreasing Markov 
pure-jump process. 
Precisely,
\begin{itemize}
 \item[(i)] given $\Phi_x(\alpha)=r$, 
 the jump rate is $\pi\capa(\B(1)\cup\B(x,r))$,
 and the process jumps to the state
 $V^{(x,r)}<r$, where~$V^{(x,r)}$ is 
a random variable with distribution
\[
 \IP[V^{(x,r)}<s] 
    = \frac{\capa(\B(1)\cup\B(x,s))}{\capa(\B(1)\cup\B(x,r))}, 
\]
for $r>s>{\rm dist}(x, \B(1)^\complement )= (1-\|x\|)^+$.

 \item[(ii)] given $\Phi_0(\alpha)=s>1$, the jump rate 
 is $2\ln s$,
 and the process jumps to the state $s^U$, where~$U$ is 
 a Uniform$[0,1]$ random variable. 
\end{itemize}
\end{theo} 
In view of the above, we consider the time-changed process $Y$, which will appear as one of the central objects of this paper,
\begin{equation}  \label{def:Y}
Y(\beta) = \beta + \ln \ln \Phi_0(e^\beta) +  \ln 2 \;,\qquad \beta 
\in \R.
\end{equation}
\begin{theo}
\label{t_Y_process}
 The process $Y$ is a stationary Markov process with unit drift 
in the positive direction, 
jump rate~$e^y$ and jump distribution given by the negative of 
an $\text{\rm{Exp}}(1)$ random variable. It solves 
 the stochastic differential equation 
\begin{equation}
\nn 
d Y(\beta) = d \beta - {\mathcal E}_\beta dN(\beta) ,
\end{equation}
where $N$ is a point process with stochastic 
intensity $\exp Y(\beta)$ and 
marks ${\mathcal E}_\beta$ with $\text{\rm{Exp}}(1)$-distribution.
 Its infinitesimal generator is given on $C^1$ functions
 $f: \R \to \R$ by
\begin{equation}\nn
{\mathcal L}f(y)= f'(y) + e^y \int_0^{+\8} 
\big[f(y-u)-f(y)\big] e^{-u} du\;.
\end{equation}
Its invariant measure is the negative of a standard Gumbel distribution, it has density $\exp\{y-e^y\}$ on $\R$.
\end{theo}
Note that the invariant measure is in agreement with \eqref{Phi_0_exact}, since the negative of logarithm of an exponentially distributed random variable is a Gumbel.

\begin{rem}
The process~$Y$ relates to models for dynamics of 
TCP (Transmission Control Protocol) for the internet:
In one of the popular congestion control protocols, 
known as MIMD (Multiplicative Increase Multiplicative Decrease), 
 the throughput (transmission flow, denoted by $X$) 
is linearly increased as long as no loss occurs in the 
transmission, and divided by 2 when a collision is detected. 
 The collision rate is proportional (say, equal) to the throughput.
Following~\cite[Eq.\ (3)]{baccelli09}, 
$X(t)$ solves $d X= Xd t -  (X/2) dM(t)$ with~$M$ 
a point process with stochastic rate~$X$.
Thus, if at every collision the throughput would be divided by the exponential of an independent exponential variable (instead of by 2), then the 2 models would be related by $X=e^Y$. 
The authors of  \cite{baccelli07, baccelli09} analyse the system, proving existence of the equilibrium, 
 formulas for moments and density using Mellin transform. 
The explicit (Gumbel) solution in the case of exponential 
jumps seems to be new.
\end{rem}

The asymptotics of the process $\Phi_x$ for $x\neq 0$ is remarkable.
First, note that $\Phi_x(\alpha) \to (1-\|x\|)^+$  a.s.\ as $\alpha\to\infty$.
Hence, we will study the process $\Phi_x$ under different scales and a common exponential time-change, depending on  $x$ being outside   the unit circle, on the circle or inside. Define, for $\beta \in \R$,
\begin{align} 
\label{def:Y_^process}
Y_x^{out}(\beta) &= \beta - \ln |\ln \Phi_x(e^\beta)| 
+ \ln \big({2}\ln^2 \|x\|\big), \quad \text{ for } \|x\|>1,  \nn \\
Y_x^{\6}(\beta) &= \beta + 2  \ln \Phi_x(e^\beta),  \quad \text{ for } 
 \|x\|=1,\phantom{\sum^B} \\
Y_x^{in}(\beta) &= \beta +  \frac{3}{2} \ln \big(\Phi_x(e^\beta)-1+\|x\|\big) + \ln \frac{3\pi}{4 \sqrt 2} \sqrt{ \frac{\|x\|}{1-\|x\|}}, \quad \text{ for }
\|x\| \in (0,1),  \nn 
\end{align}
 with the convention in the first line 
 $ \ln |\ln \Phi_x(e^\beta)|  = -\8$ when $  \Phi_x(e^\beta)=1$. 
 The following result describes 
 the  behavior of $\Phi_x(\alpha)$ for large $\alpha$:
\begin{theo}
\label{t_Y_^process} Let $Y(\cdot)$ denote the \emph{stationary} 
process defined in \eqref{def:Y}.
As $\beta_w \to +\8$,
we have
\[
\left.
\begin{array}{ll}
 {\rm for\ }   x \notin \B(1),&Y_x^{out}(\beta_w + \cdot)    \\
  {\rm for\ }    \|x\| =1, & Y_x^{\6}(\beta_w + \cdot)  \\
 {\rm for\ }      \|x\| \in(0,1), &Y_x^{in}(\beta_w + \cdot)  
\end{array}
\right\} \longrightarrow Y(\cdot)
\]
where the convergence holds in law in the Skorohod space ${\mathbb D}(\R^+; \R)$.
\end{theo}
That is, the large $\alpha$ behavior of $\Phi_x(\alpha)$
 has three different regimes according to $|x|$ being ouside, 
 on, or inside the unit circle.
  Although the scalings are different in all these regimes,
 surprisingly enough, the scaling limit is the same process $Y$.   
At the present moment, the authors have no heuristic 
explanation of why such a result holds.

Again, we observe that the invariant measure of the limit $Y$ fits with the asymptotic of the marginal distribution of $\Phi_x(\alpha)$ in Theorem \ref{t_freedisk} below. 
Our last theorem is a finer result, in the sense that $x$ does not need to be fixed.
We obtain the asymptotic law of $\Phi_x(\alpha)$
for~$x$ such that $\|x\|\geq 1$, in the regime when the 
number of trajectories which are ``close'' to~$x$
goes to infinity.

\begin{theo}
\label{t_freedisk}
For any $s>0$ and $x\notin\B(1)$
it holds that
\begin{equation}
\label{eq_freedisk}
  \IP\Big[\frac{2\alpha\ln^2\|x\|}{\ln(\Phi_x(\alpha)^{-1})}>s\Big] 
= e^{-s}\Big(1+O\Big(
\textstyle\frac{s(|\ell_x|+\ln\|x\|)}{\alpha\ln^2\|x\|}
+
\frac{\ln\theta_{x,r_s}}{\theta_{x,r_s}}\big(
\frac{1}{\ln(\|x\|+1)}+\frac{1}{\ln\theta_{x,r_s}
+\ln\|x\|}\big)\Big)\Big) ,
\end{equation}
where $r_s=\exp(-2s^{-1}\alpha\ln^2\|x\|)$
and $\theta_{x,r_s}=\frac{\|x\|-1}{r_s}$.
For $x\in \partial\B(1)$ and~$s>0$, it holds that
\begin{equation}
\label{eq_freedisk_boundary}
\IP\big[\alpha (\Phi_x(\alpha))^2>s\big] 
 = e^{-s}\big(1+O\big(\big(
\textstyle\frac{s}{\alpha}\big)^{3/2}\big)\big).
\end{equation}
\end{theo}
In particular, \eqref{eq_freedisk} implies that
 $2\alpha \ln^2\|x\| /\ln (\Phi_x(\alpha))^{-1}$
converges in distribution to an Exponential random 
variable with rate~$1$, either with fixed~$x$ and as $\alpha \to \infty$,
or for a fixed~$\alpha$ and $x\to\infty$;
also, \eqref{eq_freedisk_boundary} implies that
$\alpha (\Phi_x(\alpha))^2$
converges in distribution to an Exponential random 
variable with rate~$1$, as $\alpha \to \infty$.
Informally, the above means that,
if~$\alpha>0$ and~$x\notin \B(1)$ are such that
$\alpha \ln^2\|x\|$ is large, 
then $\Phi_x(\alpha)$ is approximately
$\exp(-2\alpha Y^{-1} 
 \ln^2\|x\|)$, 
where~$Y$ is an Exponential(1) random variable.
In the case $x\in\partial\B(1)$, $\Phi_x(\alpha)$ is 
approximately~$\sqrt{\frac{Y}{\alpha}}$
as $\alpha\to\infty$.

\section{Some auxiliary facts}
\label{s_aux}

In many computations we meet the mean logarithmic distance of $x \in \R^2$ to the points of the unit circle,
\begin{equation}
\label{df_g_x2}
 g_{x} := \frac{1}{2\pi} \int_0^\pi 
 \ln\big(\|x\|^2\!+\!1\!-\!2\|x\|\cos\theta\big) \, d\theta = \int_{\partial \B(1)} \ln \|x\!-\!z\|\; d \hm_{\B(1)}(z);
\end{equation}
that is, $g_x$ is equal to the logarithmic potential 
generated at $x$ by the harmonic measure on the disk. Compare with \eqref{df_ell_y}.
This integral can be computed:
\begin{prop} 
\label{p_prop_gx2} 
We have
\begin{equation}
g_x = 
\begin{cases}
 0,  &  \text{for } x \in \B(1),  \\
\ln \|x\|,  & \text{for }  x \notin \B(1).
\end{cases}
\end{equation}
\end{prop}
For completeness, we give a short elementary proof in the Appendix.
  The reader is referred to the Frostman's 
theorem~\cite[Theorem~3.3.4]{Ransford95} 
for how the result relates to general potential theory. 
Moreover, we mention that another proof is possible,
observing that both sides are solutions of
 $\Delta u= 2 \pi \hm_{\B(1)}$ on~$\R^2$ in the 
 distributional sense.

\subsection{On hitting and exit probabilities for~$W$ and~$\widehat W$}
\label{s_aux_hitting}
First, we recall a couple of basic facts for the exit
probabilities of the two-dimensional Brownian motion.
The following is a slight sharpening of~\eqref{hitting_BM}:
\begin{lem}
\label{l_exit_disks}
For all $x, y \in \R^2$ and $R>r>0$ with 
$x \in \B(y,R)\setminus \B(r)$, $\|y\|\leq R-2$, we have
\begin{equation}
\label{nothit_r}
 P_x\big[\tau(r)>\tau(y,R)\big] 
   = \frac{\ln(\|x\|/r)}{\ln(R/r) 
   + O\big(\frac{\|y\|\vee 1}{R}\big)}\;, 
\end{equation}
as $R\to \infty$. 
\end{lem}

\begin{proof}  It is a direct consequence of the optional stopping theorem for the local martingale $\ln \|W_t\|$ 
and the stopping time $\tau(y,R)\wedge \tau(r)$.
\end{proof}

We need to obtain some formulas for hitting probabilities by Brownian motion
of arbitrary 
compact
sets, which are ``not far'' from the origin.
Let~$\mu_r$ be the uniform probability measure on~$\partial \B(r)$;
observe that, by symmetry, $\mu_r=\hm_{\B(r)}$.
Let $\nu_{A,x}$ the entrance
measure to~$A$ starting at $x\in \B(y,R)\setminus A$;
also, let $\nu_{A,x}^{y,R}$ be the conditional entrance
measure, given that $\tau(A)<\tau(y,R)$ 
(all that with respect to the standard 
two-dimensional Brownian motion). 
\begin{lem}
\label{l_entr_harmonic}
Assume also that $A\subset\B(r)$ for some $r>0$.
We have 
\begin{align}
 \Big|\frac{d \nu_{A,x}}{d\hm_A}-1\Big| &= O\Big(\frac{r}{s}\Big),
 \label{uRN_entrance}\\
\intertext{and}
 \Big|\frac{d \nu_{A,x}^{y,R}}{d\hm_A}-1\Big|
 &= O\Big(\frac{r}{s}\Big),
 \label{RN_entrance}
\end{align}
where  $s=\|x\|-r$.
\end{lem}

\begin{proof}
 Observe that we can assume that $r=1$, the general case 
 then follows from a rescaling argument. Next, it is well known
(see e.g.\ Theorem~3.44 of~\cite{MP10})
 that for $x\notin \B(1)$ and~$y\in \partial\B(1)$ 
\begin{equation}
\label{Poisson_kernel}
 H(x,y) = \frac{\|x\|^2-1}{2\pi\|x-y\|^2}
\end{equation}
is the Poisson kernel on $\R^2\setminus \B(1)$
(i.e., the entrance measure to~$\B(1)$ starting from~$x$).
A straightforward calculation implies that 
\begin{equation}
\label{estim_kernel}
 \Big|H(x,y) - \frac{1}{2\pi}\Big| = O\Big(\frac{1}{\|x\|-1}\Big)
\end{equation}
uniformly in $y\in \partial\B(1)$.
Recall that~$\mu$ denotes the uniform measure on~$\partial\B(1)$
and $\mu=\hm_{\B(1)}$ by symmetry. 
Therefore, it holds that 
\[
 \hm_A(u) = \int_{\partial\B(1)}\nu_{A,y}(u)\, d\mu(y);
\]
also,
\[
\nu_{A,x}(u) = \int_{\partial\B(1)}\nu_{A,y}(u)\, d\nu_{\B(1),x}(y)
 = \int_{\partial\B(1)}\nu_{A,y}(u)
   \frac{d\nu_{\B(1),x}}{d\mu}(y)\, d\mu(y).
\]
Now, \eqref{estim_kernel} implies that 
$\big|\frac{d\nu_{\B(1),x}}{d\mu}-1\big|=O\big(\frac{1}{\|x\|-1}\big)$,
which shows~\eqref{uRN_entrance} for~$r=1$ and so (as observed before)
for all~$r>0$. The corresponding fact~\eqref{RN_entrance}
for the conditional entrance measure then follows 
in a standard way, see e.g.\ the calculation~(31)
in~\cite{CPV16}.
\end{proof}

We also need an estimate on the (relative) difference
of entrance measures to~$\B(1)$ from two \emph{close}
points $x_1,x_2\notin \B(1)$.
Using~\eqref{Poisson_kernel}, it is 
elementary to obtain that
\begin{equation}
\label{H_RN}
 \Big|\frac{dH(x_1,\cdot)}{dH(x_2,\cdot)} - 1\Big|
  = O\Big(\frac{\|x_1-x_2\|}{\dist\big(\{x_1,x_2\},\B(1)\big)}\Big). 
\end{equation}

Next, recall the definition~\eqref{df_ell_y} of the quantity~$\ell_x$.
Clearly, it is straightforward to obtain that
\begin{equation}
\label{ell_x_infty}
  \ell_x = \big(1+O(\|x\|^{-1})\big) \ln \|x\|,  \qquad
\text{ as }\|x\|\to\infty .
\end{equation}
We also need to know the asymptotic behaviour 
of~$\ell_x$ as $\|x\|\to 1$. 
Write
\[
 \ell_x = \ln (\|x\|-1) + 
 \frac{1}{2\pi}\int_{\partial \B(1)}
   \frac{\|x\|^2-1}{\|x-z\|^2}
\ln\frac{\|x-z\|}{\|x\|-1}\, dz;
\]
it is then elementary to obtain that there exists~$C>0$
such that
\[
0\leq \frac{\|x\|^2-1}{\|x-z\|^2}
\ln\frac{\|x-z\|}{\|x\|-1} \leq C
\]
for all $z\in\partial \B(1)$.
This means that
\begin{equation}
\label{ell_x_1}
  \ell_x = \ln (\|x\|-1)  + O(1) \qquad
\text{ as }\|x\|\downarrow 1.
\end{equation}

Now, we need an expression for the probability
that the diffusions~$W$ and~$\hW$ visit a set before 
going out of a (large) disk.
\begin{lem}
\label{l_cond_avoid_A}
 Assume that $A\subset \R^2$ is such that 
$\B(1)\subset A \subset \B(r)$.
Let $y,R$ be such that $\B(2r)\subset \B(y,R)$. Then,
for all $x \in \B(y,R)\setminus \B(2r)$ we have
\begin{align}
\IP_x\big[\tau(y,R)<\tau(A)\big] & = 
\frac{\ln \|x\|-\frac{\pi}{2}\capa(A)+O\big(\frac{\|y\|\vee 1}{R}
+\frac{r\ln r }{\|x\|}\big)}
{\ln R-\frac{\pi}{2}\capa(A) 
+ O\big(\frac{\|y\|\vee 1}{R}+\frac{r\ln r }{\|x\|}\big)},
\label{uncond_avoid_A}
\intertext{and}
 \IP_x\big[\htau(y,R)<\htau(A)\big] & = 
\frac{\ln \|x\|-\frac{\pi}{2}\capa(A)+O\big(\frac{\|y\|\vee 1}{R}
+\frac{r\ln r }{\|x\|}\big)}
{\ln R-\frac{\pi}{2}\capa(A) 
+ O\big(\frac{\|y\|\vee 1}{R}+\frac{r\ln r }{\|x\|}\big)}
\nonumber\\ 
 &\qquad\qquad \times\frac{\ln R + 
O\big(\frac{\|y\|\vee 1}{R}\big)}{\ln \|x\|}.
 \label{cond_avoid_A}
\end{align}
\end{lem}
Note that \eqref{uncond_avoid_A} deals with more general sets than \eqref{nothit_r} --
for which $A=\B(r)$ and $ \capa(A)=(2/\pi)\ln r$ -- but has different error terms.
\begin{proof}
For $x \in \B(y,R)\setminus \B(2r)$, abbreviate
(cf.\ Lemma~\ref{l_exit_disks})
\begin{equation}
\label{p_1_explicit}
p_1 = \IP_x[\tau(1)<\tau(y,R)] 
   = 1 - \frac{\ln \|x\|}{\ln R + O\big(\frac{\|y\|\vee 1}{R}\big)},
\end{equation}
and
\[
 p_A = \IP_x[\tau(A)<\tau(y,R)].
\]
Using~Lemma~\ref{l_exit_disks} and~\eqref{RN_entrance} again, 
using also that $\B(1) \subset A$, we write 
\begin{align*}
 p_A &= p_1 + \IP_x[\tau(A)<\tau(y,R)<\tau(1)]\\
&= p_1 + p_A \int_{A} \IP_v[\tau(y,R)<\tau(1)] \, d\nu_{A,x}^{y,R}(v)\\
 &= p_1 + p_A \int_{A}
 \frac{\ln\|v\|}{\ln R + O\big(\frac{\|y\|\vee 1}{R}\big)}\, 
d\nu_{A,x}^{y,R}(v)\\
 &= p_1 + \Big(1+O\Big(\frac{r}{\|x\|}\Big)\Big)
   p_A \int_{A} \frac{\ln\|v\|}{\ln R + 
O\big(\frac{\|y\|\vee 1}{R}\big)}\, d\hm_A(v)\\
  &= p_1 + \frac{\pi}{2} \Big(1+O\Big(\frac{r}{\|x\|}\Big)\Big)
     \frac{p_A}{\ln R + 
O\big(\frac{\|y\|\vee 1}{R}\big)} \capa(A),
\end{align*}
which implies that
\begin{equation}
\label{expr_p_A}
 p_A = \Big(1 - \frac{\ln \|x\|}{\ln R + 
O\big(\frac{\|y\|\vee 1}{R}\big)}\Big)
    \Bigg(1-\frac{\pi}{2}\Big(1+O\Big(\frac{r}{\|x\|}\Big)\Big)
     \frac{\capa(A)}{\ln R + 
O\big(\frac{\|y\|\vee 1}{R}\big)} \Bigg)^{-1}.
\end{equation}
Since $\IP_x[\tau(y,R)<\tau(A)]=1-p_A$,
we obtain~\eqref{uncond_avoid_A}
after some elementary calculations. 

Next, using Lemma~\ref{l_conn_W_hW}, we obtain
\begin{align*}
 \IP_x\big[\htau(y,R)<\htau(A)\big] & = \IP_x\big[\tau(y,R)<\tau(A)
   \mid \tau(y,R)<\tau(1)\big] \\
 &= \frac{1-p_A}{1-p_1}
\end{align*}
where both terms can be estimated by \eqref{p_1_explicit} and     \eqref{expr_p_A}.
Again, after some elementary calculations,
we obtain~\eqref{cond_avoid_A}.
\end{proof}

Setting $y=0$ and sending~$R$ to infinity in~\eqref{cond_avoid_A},
we derive the following fact:
\begin{cor}
\label{c_cond_avoid_A_forever}
 Assume that $A\subset \R^2$ is such that 
$\B(1)\subset A \subset \B(r)$.
 Then for all 
$x \notin \B(2r)$ 
it holds that 
 \begin{equation}
\label{avoid_A_forever}
 \IP_x[\htau(A)=\infty] =  1-
 \frac{\pi\capa(A)}{2\ln\|x\|}\Big(1 
 +O\Big(\frac{r}{\|x\|}\Big)\Big).
\end{equation}
\end{cor}

It is interesting to observe that~\eqref{expr_p_A} implies
that the capacity is translationary invariant (which is 
not very evident directly from~\eqref{df_Br_cap}), that is,
if $A\subset\R^2$ and $y\in \R^2$ are such that 
$\B(1) \subset A\cap (y+A)$, then $\capa(A)=\capa(y+A)$.
Indeed, it clearly holds that
$\IP_x[\tau(R)<\tau(A)]=\IP_{x+y}[\tau(y,R)<\tau(y+A)]$ for any $x,R$;
on the other hand,
if we assume that $\capa(A)\neq \capa(y+A)$, then
the expressions~\eqref{expr_p_A} for the two probabilities
will have different asymptotic behaviors as
$R\to\infty$ for, say, $x=R^{1/2}$, thus 
leading to a contradiction. 

Next, we relate the probabilities of certain events
for the processes~$W$ and~$\hW$. 
 In the next
result\footnote{which is analogous to Lemma~3.3~(ii) from~\cite{CPV16}}
we show that the excursions of~$W$ and~$\hW$ on a 
``distant'' (from the origin) set are almost indistinguishable:
\begin{lem}
\label{l_relation_W_hatW}
 Assume that~$M$ is compact
and suppose that $\B(1)\cap M=\emptyset$, denote $s=\dist(0,M)$,
$r=\diam(M)$, and assume that $r<s$. Then, for any $x \in M$,
\begin{equation}
\label{eq_relation_S_hatS2}
\Big\|\frac{d \IP_x\big[\hW_{\vert [0,\htau(\partial M)]} \in \cdot\,\big]}{d
\IP_x\big[W_{\vert [0,\tau(\partial M)]}\in \cdot\,\big]} -1
\Big\|_\infty = O\Big(\frac{r}{s\ln s}\Big).
\end{equation}
\end{lem}

\begin{proof}
First note that $\tau (\partial M) < \tau (1)$ and is finite $\IP_x$-a.s..

Let $A$ be a Borel set of paths starting at~$x$ and ending 
on the first hitting of~$\partial M$. 
Let~$R$ be such that $M \subset \B(R)$. 
Then, using Lemma~\ref{l_conn_W_hW}, Markov property  and~\eqref{hitting_BM},
one can write
\begin{align*}
\IP_x\big[\hW_{|[0,\htau(\partial M)]}\in A \big] &=
\IP_x\big[W_{|[0,\tau(\partial M)]}\in A \mid 
 \tau(R)<\tau(1)    \big] \\
&= \frac{ \int_{\partial M} \IP_x\big[W_{|[0,\tau(\partial M)]}
\in A, \tau(R)<\tau(1), W_{\tau(\partial M)} \in dy    \big] 
}{\IP_x\big[ \tau(R)<\tau(1)    \big]}\\
&=\frac{ \int_{\partial M} \IP_x\big[W_{|[0,\tau(\partial M)]}\in A, W_{\tau(\partial M)} \in dy \big] \IP_y\big[ \tau(R)<\tau(1) \big]}
{\IP_x\big[ \tau(R)<\tau(1)    \big]}\\
&=
 \E_x \Big(\1{\{W_{|[0,\tau(\partial M)]}\in A\}}  
 \frac{\ln\|W_{\tau(\partial M)}\|}{\ln\|x\|} \Big).
\end{align*}
 Thus, we derive the following expression for the 
Radon-Nikodym derivative from~\eqref{eq_relation_S_hatS2},
extending \eqref{def:hat_W}:
\begin{equation}\label{eq:densityShanghai}
\frac{d \IP_x\big[\hW_{|[0,\htau(\partial M)]} \in \cdot\,\big]}{d
\IP_x\big[W_{|[0,\tau(\partial M)]}\in \cdot\,\big]} =  \frac{\ln\|W_{\tau(\partial M)}\|}{\ln\|x\|} .
\end{equation}
The  desired result now follows
by writing $\ln\|y\| = \ln\|x\| + \ln\big(1
 +\frac{\|y\|-\|x\|}{\|x\|}\big)$ for $y\in\partial M$. 
\end{proof}

Let us state several other
general estimates, for the probability
of (not) hitting a given set which is, typically in (i) and (ii),
far away from the origin (so this is not related 
to Lemma~\ref{l_cond_avoid_A} 
and Corollary~\ref{c_cond_avoid_A_forever}, where $A$ was assumed to contain $\B(1)$):
\begin{lem}
\label{l_escape_from_disk}
Assume that $x\notin \B(y,r)$ and $\|y\|>r+1$
(so $\B(1)\cap \B(y,r)=\emptyset$). 
Abbreviate also $\Psi_0=\|y\|^{-1}r$, $\Psi_1=\|y\|^{-1}(r+1)$, 
$\Psi_2=\frac{r\ln r}{\|y\|}$,
$\Psi_3=r\ln r\big(\frac{1}{\|x-y\|}
+\frac{1}{\|y\|}\big)$, $\Psi_4=\frac{\|x-y\|\ln (\|y\|-1)}{\|y\|-1}$,
$\Psi_5=\frac{r(1+|\ln r^{-1}|+\ln\|y\|)}{\|y\|}$,
$\Psi_6 = \frac{r}{\|y\|-1}\ln \frac{\|y\|-1}{r}$,
and
$\Psi_7 = \frac{1}{\|x\|\ln \frac{\|y\|-1}{r}}
\big(|\ln r^{-1}|+\ln \|y\|(1+|\ln r^{-1}|+|\ln(\|y\|-1)|)\big)$.
\begin{itemize}
\item[(i)] It holds that 
\begin{equation}
\label{eq_escape_from_disk}
 \IP_x\big[\htau(y,r)<\infty\big]
= \frac{\big(\ln\|y\|+O(\Psi_0)\big)
\big(\ln\|x\|+\ln\|y\|-\ln\|x-y\|+O(\Psi_1)\big)}
{\ln\|x\|\big(2\ln\|y\|-\ln r+O(\Psi_1)\big)}.
\end{equation}
\item[(ii)] For any $r>1$ and any 
set  $A$ 
such that $\B(y,1)\subset A\subset \B(y,r)$, we have
\begin{equation}
\label{eq_escape_from_anyset}
 \IP_x\big[\htau(A)<\infty\big]
= \frac{\big(\ln\|y\|+O(\Psi_0)\big)\big(\ln\|x\|+\ln\|y\|-\ln\|x-y\|
+O(\Psi_3)\big)}
{\ln\|x\|\big(2\ln\|y\|-\frac{\pi}{2}\capa({\mathcal T}_{-y}A)+ O(\Psi_2)
\big)},
\end{equation}
being $ {\mathcal T}_{y} $ the translation by vector $y$.
\item[(iii)] Next, we consider the regime  when~$x$
and~$y$ are fixed, $r<\|x-y\|$, and $\|x-y\|$ is small
 (consequently, $r$ is small as well).
 Then, we have
\begin{equation}
 \label{eq_escape_small_dsk_close}
\IP_x\big[\htau(y,r)<\infty\big]
 = \frac{\ln\|x-y\|^{-1}+\ell_y+\ln\|x\|
+O(\Psi_4)}{\ln r^{-1}+\ell_y+\ln\|y\|+O(\Psi_3)}
 \big(1+O(\|x-y\|)\big).
\end{equation}
\item[(iv)] The last regime we need to consider is 
when~$x$ is large, but~$y$ possibly not (it can be 
even very close to~$\B(1)$).
Then, we have
\begin{equation}
 \label{eq_escape_verycloseto1}
 \IP_x\big[\htau(y,r)<\infty\big]
 =  \frac{(\ln\|y\|+O(\Psi_0))(\ln\|y\|+O(\Psi_5 + \Psi_7))}
{\ln\|x\|(\ln r^{-1}+ \ln\|y\|
+\ell_y+O(\Psi_6))}.
\end{equation}
\end{itemize}
\end{lem}
We now mention a remarkable property of the process~$\hW$, to be compared with the display just after~(36) in~\cite{CPV16} for the discrete case.
\begin{rem} 
\label{rem:1/2} 
Observe that, for all $x \in \R^2\setminus \B(1)$ and for all $r>0$,
\eqref{eq_escape_from_disk} yields
\[
\IP_x \big[\htau(y,r)<\8\big] \to \frac{1}{2}
 \qquad \text{as }  \|y\|\to \8.
\]
\end{rem}
\begin{proof}
For the parts (i)--(ii),
one can use essentially the same argument as in the proof of Lemma~3.7
of~\cite{CPV16}; we sketch it here for completeness.
Fix some $R>\max\{\|x\|,\|y\|+r\}$. Define the quantities 
(see Figure~\ref{f_p12})
\begin{align*}
 h_1 &= \IP_x[\tau(1)<\tau(R)],\\
 h_2 &= \IP_x[\tau(y,r)<\tau(R)],\\
 p_1 &= \IP_x[\tau(1)<\tau(R)\wedge
 \tau(y,r)],\\
 p_2 &= \IP_x[\tau(y,r)<\tau(R)\wedge
 \tau(1)],\\
 q_{12} &= \IP_{\nu_1}[\tau(y,r)<\tau(R)],\\
 q_{21} &= \IP_{\nu_2}[\tau(1)<\tau(R)],
\end{align*}
where $\nu_1$ is the entrance measure to~$\B(1)$ starting from~$x$
conditioned on the event 
$\{\tau(1)<\tau(R)\wedge \tau(y,r)\}$, and~$\nu_2$ 
is the entrance measure to~$\B(y,r)$ starting from~$x$
conditioned on the event 
$\{\tau(y,r)<\tau(R)\wedge \tau(1)\}$.
\begin{figure}
\begin{center}
\includegraphics{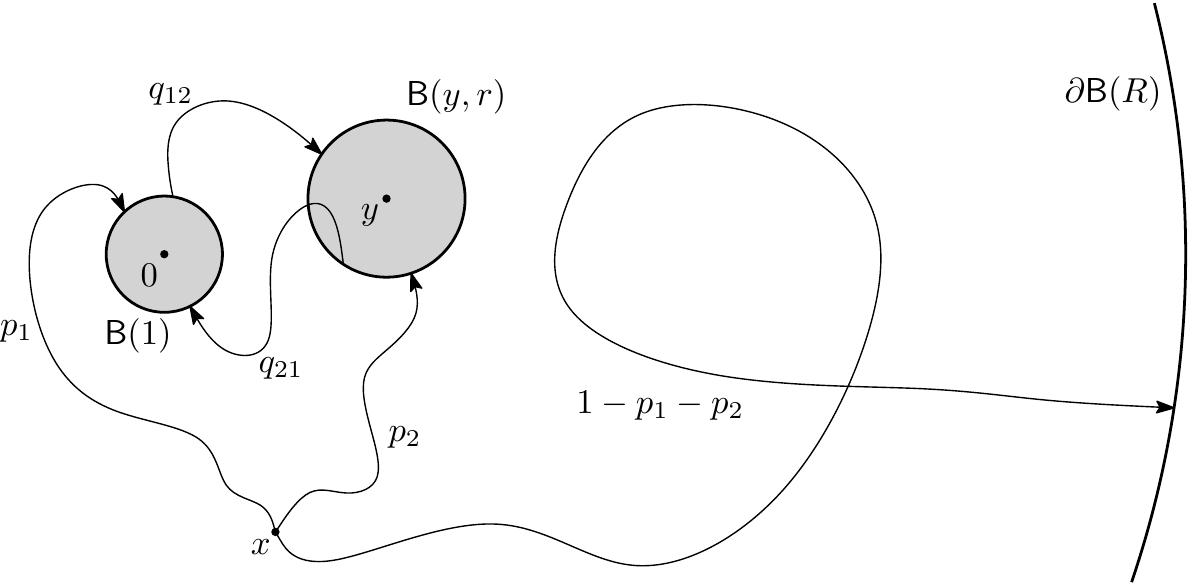}
\caption{On the proof of Lemma~\ref{l_escape_from_disk}}
\label{f_p12}
\end{center}
\end{figure}
Using Lemma~\ref{l_exit_disks}, we obtain 
\begin{align}
 h_1 &= 1-\frac{\ln\|x\|}{\ln R},
\label{expr_h1}\\
 h_2 &= 1-\frac{\ln\|x-y\|-\ln r}
 {\ln R-\ln r+O(R^{-1}\|y\|)},
\label{expr_h2}
\end{align}
and 
\begin{align}
q_{12} &= 1-\frac{\ln\|y\|
-\ln r+O(\Psi_1)}{\ln R-\ln r+O(R^{-1}\|y\|)},
\label{expr_q12} \\
q_{21} &=1-\frac{\ln\|y\|+O(\Psi_0)}{\ln R},
\label{expr_q21} 
\end{align}
using \eqref{hitting_BM} for the last line.
Then, we use the fact that, in general,
\begin{align*}
 h_1 &= p_1 + p_2q_{21},\\
 h_2 &= p_2 + p_1q_{12},
\end{align*}
   and therefore
\begin{align}
 p_1 &= \frac{h_1-h_2q_{21}}{1-q_{12}q_{21}},
\label{expr_p1}\\
 p_2 &= \frac{h_2-h_1q_{12}}{1-q_{12}q_{21}}.
\label{expr_p2}
\end{align}
We then write
\begin{equation}
\label{p_q_h}
 \IP_x[\tau(y,r)<\tau(R) \mid \tau(1)>\tau(R)]
 = \frac{p_2(1-q_{21})}{1-h_1}
  = \frac{(1-q_{21})(h_2-h_1q_{12})}{(1-h_1)(1-q_{12}q_{21})}.
\end{equation}
Note that~\eqref{expr_h1}--\eqref{expr_q21} imply that
\begin{align}
(1-h_1) \ln R &= \ln\|x\|, \label{1_h1}\\
\lim_{R\to\infty} (1-h_2) \ln R &= \ln\|x-y\|-\ln r, \label{1_h2}\\
\lim_{R\to\infty} (1-q_{21})\ln R &= \ln \|y\| + O(\Psi_0),
\label{1_q21}\\
\lim_{R\to\infty} (1-q_{12}q_{21})\ln R &= 2\ln \|y\| - \ln r
 +O(\Psi_1), \label{1_q12q21}\\
 \lim_{R\to\infty} (h_2-h_1q_{12})\ln R &= \ln\|x\|+\ln\|y\|
  -\ln\|x-y\| + O(\Psi_1), \label{h2_h1q12}\\
 \lim_{R\to\infty} (h_1-h_2q_{21})\ln R &= \ln\|x-y\|-\ln\|x\|
+\ln\|y\| - \ln r + O(\Psi_0) \label{h1_h2q21}.
\end{align}
We then plug \eqref{1_h1}--\eqref{h2_h1q12} into~\eqref{p_q_h}
and send~$R$ to infinity to obtain~\eqref{eq_escape_from_disk}.
The proof of~\eqref{eq_escape_from_anyset} is quite
analogous (one needs to use~\eqref{uncond_avoid_A} there;
note also that we indirectly assume in~(ii) that $r\geq 1$).

\medskip
\textit{Part~(iii).}
Next, we prove~\eqref{eq_escape_small_dsk_close}.
By~\eqref{H_RN}, we can write for any~$z$ such that 
$\|y-z\|\leq \|y-x\|$
\begin{equation}
\label{dHz/dHy}
 \Big|\frac{dH(z,\cdot)}{d H(y,\cdot)}-1\Big| 
 \leq O\Big(\frac{\|x-y\|}{\|y\|-1}\Big) .
\end{equation}
Then, a last-exit-decomposition argument
implies that 
\begin{equation}
\label{nu1_case4}
 \Big|\frac{d\nu_1}{d H(y,\cdot)}-1\Big| 
 \leq O\Big(\frac{\|x-y\|}{\|y\|-1}\Big) 
+ O\big((\ln R)^{-1}\big).
\end{equation}

We then write
\begin{align*}
 1-q_{12} &= \int_{\partial\B(1)} 
\frac{\ln\|y-z\|-\ln r}{\ln R-\ln r + O(R^{-1}\|y\|)}\,d\nu_1(z)\\
 &= \int_{\partial\B(1)} 
\frac{\ln\|y-z\|-\ln r}{\ln R-\ln r + O(R^{-1}\|y\|)}
\Big(O\Big(\frac{\|x-y\|}{\|y\|-1}\Big) 
+ O\big((\ln R)^{-1}\big)\Big)
H(y,dz)\\
 &= \frac{\ell_y-\ln r}{\ln R-\ln r + O(R^{-1}\|y\|)}
\Big(O\Big(\frac{\|x-y\|}{\|y\|-1}\Big) 
+ O\big((\ln R)^{-1}\big)\Big).
\end{align*}

Then, we obtain the following refinements of
\eqref{1_q12q21}--\eqref{h2_h1q12}: 
\begin{align}
 \lim_{R\to\infty} (1-q_{12}q_{21})\ln R 
&= \Big(1+O\Big(\frac{\|x-y\|}{\|y\|-1}\Big)\Big)
(\ln r^{-1} + \ell_y) +\ln \|y\| +O(\Psi_0), 
\label{1_q12q21_close}\\
 \lim_{R\to\infty} (h_2-h_1q_{12})\ln R &= \ell_y
+\ln \|x\| -\ln \|x-y\| +O(\Psi_4). 
\label{h2_h1q12_close}
\end{align}
As before, the claim then follows from~\eqref{p_q_h}
(note that $\ln\|x\| - \ln\|y\| = O(\|x-y\|)$).

\medskip
\textit{Part~(iv).}
Finally, we prove~\eqref{eq_escape_verycloseto1}.
Recalling notations introduced just before Lemma \ref{l_entr_harmonic}, we denote for short  ${\tilde \nu}_x(\cdot)=\nu_{\B(1),x}^{0,R}$ the entrance measure into $\B(1)$ starting from $x$ conditioned on $\{\tau(1)< \tau(R)\}$, and $\mu=\mu_1$ the uniform law on the circle.
Notice that~\eqref{RN_entrance} implies
\[
 {\tilde \nu}_x(\cdot) = \mu(\cdot)\big(1+O(\|x\|^{-1})\big),
\]
and, on the other hand
\[
 {\tilde \nu}_x(\cdot) = (1-p_2)\nu_1(\cdot) 
   + p_2 H(y,\cdot)\big(1+O\big(\textstyle\frac{r}{\|y\|-1}\big)\big).
\]
We thus obtain
\begin{equation}
\label{expr_nu1}
  \nu_1(\cdot) = \frac{1}{1-p_2}
 \Big(\mu(\cdot)\big(1+O(\|x\|^{-1})\big)
-p_2H(y,\cdot)\big(1+O\big(\textstyle\frac{r}{\|y\|-1}\big)\big)\Big).
\end{equation}
Using Lemma~\ref{l_exit_disks} and Proposition~\ref{p_prop_gx2} in the definition of $q_{12}$,
 we  obtain
\begin{equation}
\label{new_1-q12}
 \lim_{R\to\infty} (1-q_{12})\ln R
 = \ln r^{-1} + \frac{1}{1-p_2}
\Big(\big(1+O(\|x\|^{-1})\ln \|y\| 
 - p_2\ell_y\big(1+O\big(\textstyle\frac{r}{\|y\|-1}\big) \big)\Big). 
\end{equation}
We then recall the expression~\eqref{expr_p2} to obtain that
\[
 p_2 = \frac{O\big(\frac{\|y\|}{|x\|}\big)+\frac{1}{1-p_2}
 \big[\ln\|y\|\big(1+O(\|x\|^{-1})\big)
-p_2\ell_y\big(1+O\big(\textstyle\frac{r}{\|y\|-1}
\big)\big)\big]}{\ln\|y\|+\ln r^{-1}+O\big(\frac{r}{\|y\|}\big)
+ \frac{1}{1-p_2}
 \big[\ln\|y\|\big(1+O(\|x\|^{-1})\big)
-p_2\ell_y\big(1+O\big(\textstyle\frac{r}{\|y\|-1}
\big)\big)\big]},
\]
which can be rewritten as
\begin{equation}
\label{eq_on_p2}
 p_2 = \frac{a-b p_2}{c-d p_2},
\end{equation}
where, being ${\hat \ell} = 
\ell_y\big(1+O\big(\textstyle\frac{r}{\|y\|-1}\big)\big)$
and $w=\ln\|y\|\big(1+O(\|x\|^{-1}))$,
\begin{align*}
 a &= w + O\big(\textstyle\frac{\|y\|}{\|x\|}\big),\\
 b &= {\hat \ell},\\
 c &= w+ \ln r^{-1}+\ln\|y\|+O\big(\textstyle\frac{r}{\|y\|}\big),\\
 d &= {\hat \ell}+ \ln\|y\|+\ln r^{-1}
\end{align*}
(we need to keep track of the expressions 
which are \emph{exactly} equal, not only up to $O$'s).
We then solve~\eqref{eq_on_p2} to obtain
\[
 p_2 = \frac{(b+c)-\sqrt{(b+c)^2-4ad}}{2d}.
\]
Note also that $\ln r^{-1}+\ell_y=O(\ln\frac{\|y\|-1}{r})$.
Then, after some elementary (but long)
 computations one can obtain that
\[
\sqrt{(b+c)^2-4ad} = (\ell_y+\ln r^{-1})(1 + O(\Psi_5+\Psi_7)),
\]
which yields 
\begin{equation}
\label{new_expr_p2}
 p_2 = \frac{\ln\|y\| + O(\Psi_5+\Psi_7)}{\ln r^{-1}+\ln\|y\|
+\ell_y +O(\Psi_6)} .
\end{equation}
We use \eqref{p_q_h}, \eqref{1_h1}, 
and~\eqref{1_q21} to conclude the proof of 
Lemma~\ref{l_escape_from_disk}.
\end{proof}

\subsection{Some geometric properties of~$\hW$
and the Wiener moustache}
\label{s_geom_hW}
For two stochastic processes $X^{(1)}, X^{(2)}$, we say that they 
coincide trajectory-wise
if there exists a monotone (increasing or decreasing) stochastic
process~$\sigma$ such that a.s.\ it holds
that $X^{(1)}_{\sigma(t)}= X^{(2)}_t$ for all~$t$.

Recall that, in Remark \ref{rem_hW_b}, we denoted by~$\hW^r$ the Brownian motion
conditioned on never hitting~$\B(r)$, for $r>0$; 
also, it can be constructed as a time change of $r\hW$.
Proposition~\ref{prop:elemhW} implies
that~$\hW^r$ 
can be seen as~$\hW$ conditioned on never hitting~$\B(r)$,
i.e., 
\begin{equation}
\label{Wr-W1}
 \IP_x\big[\hW_{[0,\htau(R)]}\in \cdot \mid \htau(R)< \htau(r)\big] =
   \IP_x\big[\hW^r_{[0,\htau^r(R)]}\in \cdot\,\big]
\end{equation}
for any~$R>r$ and any~$x$ such that $r \leq \|x\| < R$ (we have denoted $\htau^r(R)=\inf\{t \geq 0: \hW^r_t \geq R\}$).

Using the above fact, we can prove the following lemma:
\begin{lem}
\label{l_W_go_in}
Let~$\zeta$ be a random point
with uniform distribution on~$\partial \B(1)$ and let us fix $r_0>1$.
For $t\geq 0$ define $V_{-t}= r_0\hW^{(1)}_t$, $V_t= \hW^{(2)}_t$,
where $\hW^{(1,2)}$ are two independent conditioned 
(on not hitting~$\B(1)$) Brownian
motions started from~$\zeta$ and~$r_0\zeta$ correspondingly.
Denote 
\[
\QQ=\inf_{t\in\R} \|V_t\| = \inf_{t>0}\|\hW^{(2)}_t\|. 
\]
Let~$\eta$ be an instance of
Wiener moustache, as in Definition~\ref{df_moustache}.
Then, for all $r_0>1$ and $h \in (1,r_0)$,  
 the law of~$h\times \eta$ is a regular version of the 
conditional law of~$V_\R$ given~$\QQ=h$.
\end{lem}
 In other words, the range~$V_{ \R}$ has same law as an
independent  product $ \QQ \times \eta$ of a Wiener moustache~$\eta$
 and a random variable~$\QQ$ distributed 
as in~\eqref{eq:dist-min}.

\begin{figure}
\begin{center}
\includegraphics{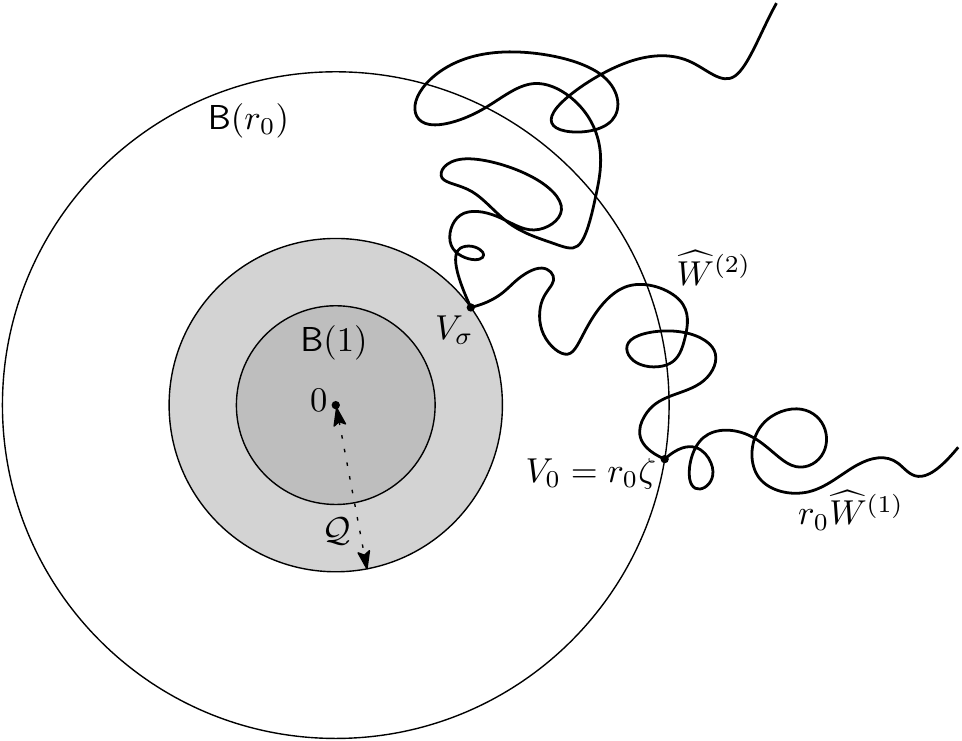}
\caption{On the proof of Lemma~\ref{l_W_go_in}}
\label{f_W_goes_in}
\end{center}
\end{figure}

\begin{proof}
The idea is to split $(V_t)_{t \geq 0}$ 
at its minimum distance from the origin. 
We use polar coordinates, and let us first study the radius 
$\RR_t=\|V_t\|$, starting with $t \in \R^+$.
 From~\eqref{escape_condBM} 
we derive the density of the minimal distance to the origin,
\begin{equation} \label{eq:dist-min}
 \QQ = \min\{ \RR_t; t \geq 0\} \sim \frac{1}{h \ln r_0} 
 \mathbf{1}_{[1,r_0]}(h) dh. 
\end{equation}
 Applying general techniques of path decomposition~\cite[Theorem 2.4]{Will74} to the one-dimensional 
diffusion~$\RR_t$, we obtain that
\begin{itemize}
\item[(a)] there a.s.\ exists a unique random time
$\sigma \in (0,\8)$ such that $\RR_{\sigma}=\QQ$;
\item[(b)] given~$\QQ$ and~$\sigma$, the process
$(\RR_t, t \in [0, \sigma])$ has the same law as
the diffusion~$\RR$ started at~$r_0$ and 
conditioned to converge to~1, 
observed up to the time of first hit of~$\QQ$;
\item[(c)]  given $\QQ, \sigma$ and $(\RR_t, t \in [0, \sigma])$,
the process $(\RR_{\sigma+t}, t \geq 0)$
has the same law as  the diffusion~$\RR$ started at~$\QQ$ and 
conditioned on staying in $[\QQ, \8)$. 
\end{itemize}
We recall that conditional diffusions are rigorously defined 
as in~\eqref{eq:BES2}, \eqref{eq:BES2cond}.
By the scaling property~(iv) of Proposition~\ref{prop:elemhW}, 
we see that the conditioned diffusion in (c) has the same law 
as an independent product $\QQ \times (\RR_t; t \geq 0)$
with $\RR_0=1$. For the  conditioned diffusion in (b), 
we use the reversibility property~(ii)  of 
Proposition~\ref{prop:elemhW}: Given~$\QQ$ and~$\sigma$, 
\[
\big(\RR_{\sigma -t}, t \in [0, \sigma]\big) 
\eqlaw \big(\RR_t', t \in [0, L_{r_0}']\big)\;,
\]
where~$\RR'$ is an independent~$\RR$-process 
(i.e., the norm of Brownian motion conditioned to be at least~$\QQ$)
starting from~$\QQ$ and where $L_{r_0}'=\sup\{t \geq 0: \RR_t'=r_0\}$ 
is the last exit time of~$\RR'$ from~$r_0$.
Pasting it with the (independent) piece for negative times 
we define a new process, 
\[
X_t=
\begin{cases}
\RR_{\sigma -t} , & t \in [0, \sigma],     \\
r_0 \big\| \hW^{(1)}_{t-\sigma}\big\|, 
\vphantom{\int\limits^{A}}& t > \sigma,
\end{cases}
\]
 which has the same law as the above process~$\RR'$. 
 By (c), the processes $(\RR_{\sigma +t}, t \geq 0) $
 and~$X$ are independent given $\QQ$. 

Now we consider the two-dimensional process~$V$. 
By rotational invariance, 
it is clear that~$\QQ^{-1}V_\sigma$
is uniformly distributed on~$\partial \B(1)$. Once the norm~$\|V_t\|$
 is defined for all~$t$, one uses an independent Brownian 
motion and the SDE~\eqref{df_Theta_t} to construct the angle process
 $(\Theta_t; t \in \R)$. Given the angle~$\Theta_\sigma$, 
it holds that $(\Theta_t; t \geq 0)$ 
and $(\Theta_{-t}; t \geq 0)$ are independent
with same distribution as in~\eqref{df_Theta_t}.  

Finally, $(V_{\sigma+t}; t \geq 0)$ and $(V_{\sigma -t}; t \geq 0)$ 
are independent conditionally on~$V_\sigma$, 
and both distributed as $\|V_\sigma\| \times \hW$.
This concludes the proof. 
\end{proof}

\subsection{Harmonic measure and capacities}
\label{s_aux_capacities}
First, we need an expression on the harmonic measure~$\hhm_A$
with respect to the diffusion~$\hW$. 
\begin{lem}
\label{l_hat_hm}
 Assume that $\B(1)\subset A$ 
and let~$M$ be a measurable subset of~$\partial A$.
We have
\begin{equation}
\label{def_hat_hm}
 {\widehat{\hm}}_A(M) = \frac{\int_M \ln\|y\| \, d\hm_A(y)}
 {\int_{\partial A} \ln\|y\| \, d\hm_A(y)},
\end{equation}
that is, $\hhm_A$ is~$\hm_A$ biased by logarithm of the distance
to the origin.
\end{lem}

\begin{proof}
Let $\nu^{R}_{A,x}:=\nu^{0,R}_{A,x}$ be the conditional entrance
measure to~$A$ starting at $x\in \B(R)\setminus A$,
given that $\tau(A)<\tau(R)$.
For $M\subset\partial A$ we can write, using
Lemma~\ref{l_conn_W_hW},
\begin{align}
 \IP_x\big[\hW_{\htau(A)}\in M\mid \htau(A)<\htau(R)\big] & = 
 \frac{\IP_x[W_{\tau(A)}\in M, \tau(A)<\tau(R)<\tau(1)]}
 {\IP_x[\tau(A)<\tau(R)<\tau(1)]}\nonumber\\
 &= \frac{\IP_x[W_{\tau(A)}\in M, \tau(R)<\tau(1) \mid \tau(A)<\tau(R)]}
 {\IP_x[\tau(R)<\tau(1)\mid \tau(A)<\tau(R)]}\nonumber\\
 &= \frac{\int_M \ln\|y\| \, d\nu^{R}_{A,x}(y)}
 {\int_{\partial A} \ln\|y\| \, d\nu^{R}_{A,x}(y)}
 \label{cond_entr_m}
\end{align}
(observe that, by~\eqref{hitting_BM}, one has to 
integrate~$\frac{\ln\|y\|}{\ln R}$ with respect to $\nu^{R}_{A,x}$,
and then the term $\ln R$ cancels). So, using~\eqref{RN_entrance} 
we obtain~\eqref{def_hat_hm}.
\end{proof}

Before proceeding, let us notice the following immediate consequence
of~\eqref{avoid_A_forever}: for any bounded $A\subset\R^2$
such that $\B(1)\subset A$, we have
\begin{equation}
\label{expr_Cap}
 \capa(A) = \lim_{\|x\|\to\infty} \frac{2}{\pi}\ln\|x\|
 \IP_x[\htau(A)<\infty].
\end{equation}

Next, we need estimates for the capacity
of a union of~$\B(1)$ with a ``distant'' set (in particular,
 a disk), and also that of a disjoint union of the unit disk and a set.
\begin{lem}
\label{l_cap_distantdisk}
  Assume that 
  $\|y\|>r+1$. 
\begin{itemize}
 \item[(i)]
We have
\begin{equation}
\label{eq_cap_distdisk}
 \capa\big(\B(1)\cup \B(y,r)\big)
  = \frac{2}{\pi}\cdot
  \frac{\ln^2\|y\| +O(\|y\|^{-1}(r+1)\ln\|y\|)}
{2\ln\|y\|-\ln r+O(\|y\|^{-1}(r+1))}.
\end{equation}
 \item[(ii)] For any set $A$ 
such that $\B(y,1)\subset A  \subset \B(y,r)$, we have 
\begin{equation}
\label{eq_cap_distantset}
\capa\big(\B(1)\cup A\big) =\frac{2}{\pi}\cdot
\capa\big(\B(1)\cup A\big) =\frac{2}{\pi}\cdot
  \frac{\ln^2\|y\| +O(\|y\|^{-1} r\ln\|y\|\ln r)}
{2\ln\|y\|-\frac{\pi}{2}\capa( {\mathcal T}_{-y} A)+O(\|y\|^{-1}r)}.
\end{equation}
with $ {\mathcal T}_{y} $ the translation by vector $y$.
 \item[(iii)]  
Moreover, we have the following refinement of~\eqref{eq_cap_distdisk}
\begin{equation} \label{eq_cap_closedisk}
 \capa\big(\B(1)\cup \B(y,r)\big)  = \frac{2}{\pi}\cdot\frac{(\ln\|y\|+O(\Psi_0))(\ln\|y\|+O(\Psi_5+\Psi_7))}
{\ln r^{-1}+\ell_y+\ln\|y\|+O(\Psi_6)}
\end{equation}
with $\Psi_0, \Psi_5, \Psi_6, \Psi_7$ as in Lemma~\ref{l_escape_from_disk}. 
%
\end{itemize}
\end{lem}

\begin{proof}
All these expressions follow from~\eqref{expr_Cap} 
and Lemma~\ref{l_escape_from_disk}.
\end{proof}

We also need to estimates the capacity of a 
union of two overlapping disks.
We first consider a (typically small) disk with center on 
the boundary of~$\B(1)$.
\begin{lem}
\label{l_cap_blister}
 For any $x\in\partial\B(1)$ and $r\in (0,2)$ it holds that
\begin{align}
\label{eq_cap_blister} 
\capa\big(\B(1)\cup \B(x,r)\big) 
&=\frac{2}{\pi} 
\ln\frac{2\sin(\pi\phi)}{(1+\phi)\sin(\pi\frac{1-\phi}{1+\phi})}\\ \nn
&= \frac{r^2}{\pi} (1+O(r))
\quad \text{ as $r\to 0$,}
\end{align}
where $\phi = \frac{2}{\pi} \arcsin \frac{r}{2}$.
\end{lem}

\begin{proof}
Clearly, without loss of generality one may assume that $x=1$;
let us abbreviate $A_r=\B(1)\cup \B(1,r)$.
 We intend to use Theorem~5.2.3 of~\cite{Ransford95};
for that, we need to find a conformal mapping~$f$
of $\C_\infty\setminus A_r$ onto $\C_\infty\setminus \B(s)$,
which sends~$\infty$ to~$\infty$ and such that 
$f(z)=z+O(1)$ as $z\to\infty$, where $s>1$
(and then Theorem~5.2.3 of~\cite{Ransford95} will 
imply that the capacities of~$A_r$ and~$\B(s)$ are equal).
Clearly, for this it is enough to map $\C_\infty\setminus A_r$ 
onto $\C_\infty\setminus \B(1)$ with~$f$ such that
$f(z)=cz+O(1)$, where $|c|\in(0,1)$, and then normalize.
That is, we then obtain that $c^{-1}f$ sends 
$\C_\infty\setminus A_r$ onto $\C_\infty\setminus \B(|c|^{-1})$,
which would give that $\capa(A_r)=\frac{2}{\pi}\ln |c|^{-1}$.

 First, it is elementary to obtain that $\partial\B(1)$
and~$\partial\B(1,r)$ intersect in the points $e^{i\pi\phi}$
and~$e^{-i\pi\phi}$. We then apply the map 
\begin{equation} \label{eq:f_1}
 f_1(z) = \frac{z-e^{i\pi\phi}}{z-e^{-i\pi\phi}},
\end{equation}
which sends the first of these points to~$0$ and the second
to~$\infty$. Observe also that $f_1(\infty)=1$.
Since~$f_1$ is a M\"obius transformation, it sends 
the two arcs that form $\partial A_r$ into two rays,
and therefore $\C_\infty\setminus A_r$ gets mapped
to a sector. The tangent vectors $v_{1,2}$ to the two arcs
at~$e^{i\pi\phi}$ have directions 
$\frac{\pi}{2}+\pi\phi$ and~$\frac{\pi\phi}{2}$ 
(see Figure~\ref{f_disk_bump}),
so the angle of that sector is 
$\frac{\pi}{2}+\pi\phi-\frac{\pi\phi}{2}=\frac{\pi}{2}(\phi+1)$.
\begin{figure}
\begin{center}
\includegraphics{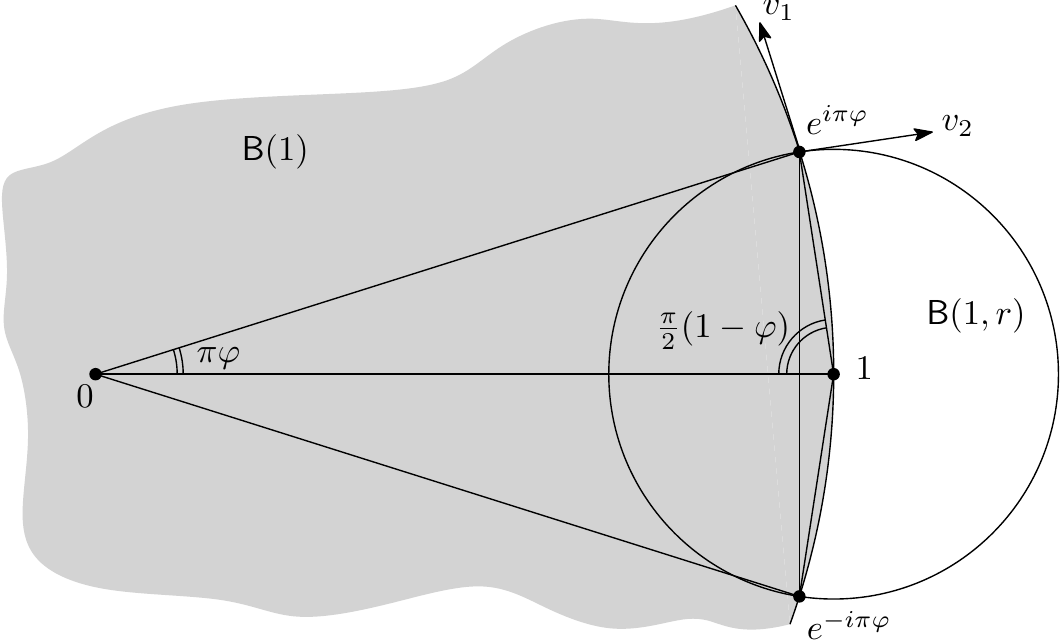}
\caption{The direction of vector~$v_1$ is $\frac{\pi}{2}+\pi\phi$,
and the direction of vector~$v_2$ is $\frac{\pi\phi}{2}$.}
\label{f_disk_bump}
\end{center}
\end{figure}

To see what are the directions of these two rays, 
observe that 
\[
 f'_1(z) = \frac{e^{i\pi\phi}-e^{-i\pi\phi}}{(z-e^{-i\pi\phi})^2},
\]
so $f'_1(e^{i\pi\phi})= (e^{i\pi\phi}-e^{-i\pi\phi})^{-1}
= \frac{-i}{2\sin \pi\phi}$.
This means that, at point~$e^{i\pi\phi}$, the map~$f_1$
rotates~$\frac{\pi}{2}$ clockwise, and so the direction 
of the first ray is $\pi\phi$ and the direction
of the second one is $(-i\frac{\pi}{2}(1-\phi))$.

Next, we apply the function 
\[
z\mapsto \big(e^{i\frac{\pi}{2}(1-\phi)} z\big)^{\frac{2}{1+\phi}}
\]
 to send this sector to the upper half-plane
(we first rotate the sector so that its right boundary goes
to the positive part of the horizontal axis, and then
``open'' it so that its angle changes from~$\frac{\pi}{2}(\phi+1)$
to~$\pi$).
The point~$1$ (the image of~$\infty$ under~$f_1$)
is then sent to $e^{i\pi\frac{1-\phi}{1+\phi}}$.
Then, we need to map the upper half plane 
to~$\C_\infty\setminus \B(1)$ in such a way that 
$e^{i\pi\frac{1-\phi}{1+\phi}}$ is sent back to~$\infty$.
 This is achieved by the function 
\[
 z\mapsto \frac{z-e^{-i\pi\frac{1-\phi}{1+\phi}}}
{z-e^{i\pi\frac{1-\phi}{1+\phi}}}.
\]

Gathering the pieces, we then obtain that the map 
\begin{equation}
\label{final_map}
 f(z) =  \frac{\big(e^{i\frac{\pi}{2}(1-\phi)}\frac{z-e^{i\pi\phi}}
{z-e^{-i\pi\phi}}\big)^{\frac{2}{1+\phi}}
-e^{-i\pi\frac{1-\phi}{1+\phi}}}
{\big(e^{i\frac{\pi}{2}(1-\phi)}\frac{z-e^{i\pi\phi}}
{z-e^{-i\pi\phi}}\big)^{\frac{2}{1+\phi}}
-e^{i\pi\frac{1-\phi}{1+\phi}}}
\end{equation}
sends $\C_\infty\setminus A_r$ to $\C_\infty\setminus \B(1)$.

Next, we need to obtain the asymptotic behavior of the
above function as $z\to \infty$. 
First, it clearly holds that 
\begin{equation}
\label{modnum}
 \Big|\Big(e^{i\frac{\pi}{2}(1-\phi)}\frac{z-e^{i\pi\phi}}
{z-e^{-i\pi\phi}}\Big)^{\frac{2}{1+\phi}}
-e^{-i\pi\frac{1-\phi}{1+\phi}}\Big| 
\to  \Big|e^{i\pi\frac{1-\phi}{1+\phi}}
-e^{-i\pi\frac{1-\phi}{1+\phi}}\Big| 
 = 2\sin\big(\pi \textstyle\frac{1-\phi}{1+\phi}\big)
\quad \text{ as }z\to\infty.
\end{equation}
We have 
\[
 \frac{z-e^{i\pi\phi}}{z-e^{-i\pi\phi}}
= \frac{1-e^{i\pi\phi}z^{-1}}{1-e^{-i\pi\phi}z^{-1}}
 = 1-(e^{i\pi\phi}-e^{-i\pi\phi})z^{-1}+O(z^{-2})
=1-2i\sin(\pi\phi) z^{-1}+O(z^{-2}),
\]
so, we can write
\begin{align}
\lefteqn{
 \lim_{z\to\infty} |z|\cdot
\Big|\Big(e^{i\frac{\pi}{2}(1-\phi)}\frac{z-e^{i\pi\phi}}
{z-e^{-i\pi\phi}}\Big)^{\frac{2}{1+\phi}}
-e^{i\pi\frac{1-\phi}{1+\phi}} \Big|
}
\nonumber\\
 &=\lim_{z\to\infty} |z|\cdot 
 \Big|\big(1-2i\sin(\pi\phi) z^{-1}+O(z^{-2})\big)^{\frac{2}{1+\phi}}
-1\Big|\nonumber\\
&= \frac{4\sin(\pi\phi)}{1+\phi}.
\label{moddenom}
\end{align}
%
%
The relations~\eqref{modnum}--\eqref{moddenom} imply that 
the function in~\eqref{final_map} is $f(z)=cz+O(1)$ with
\[
 |c| =  2\sin\big(\pi {\textstyle\frac{1-\phi}{1+\phi}}\big)
  \Big( \frac{4\sin(\pi\phi)}{1+\phi}\Big)^{-1} = 
\frac{(1+\phi)\sin(\pi \frac{1-\phi}{1+\phi})}{2\sin(\pi\phi)}.
\]
This gives the exact formula in~\eqref{eq_cap_blister},
and then one obtains the asymptotic expression as
$r\to 0$ with a straightforward calculation.
\end{proof}

%

%

Also, we treat the case when the center of the second disk 
is in the interior of the first one.
\begin{lem}
\label{l_cap_interior}
 For  $\|x\| \in (0,1)$ and $r \in (1-\|x\|,1+\|x\|)$ it holds that
\begin{align}
\label{eq_cap_interior} 
\capa\big(\B(1)\cup \B(x,r)\big) 
&=\frac{2}{\pi} 
\ln\frac{\sin(\pi\phi)}{(1 \! + \! \phi \! - \! \psi)\sin\frac{\pi \phi}{1 \! + \! \phi \! - \! \psi}}\\ \nn
&= \frac{4 \sqrt 2}{3\pi^2} \sqrt{\frac{ 1 \! - \! \|x\|}{\|x\|}} h^{3/2} +O(h^2)
\end{align}
as $h:=r -(1-\|x\|)  \searrow 0$, where 
\[
\phi = \frac{2}{\pi} \arcsin \Bigg(\frac{1}{2} 
\sqrt{\frac{r^2-(1-\|x\|)^2}{\|x\|}}\Bigg)\;,
\qquad \psi=\frac{2}{\pi} \arcsin \Bigg(  \frac{1}{2} 
\sqrt{\frac{(r+\|x\|)^2-1}{r\|x\|}}  \Bigg)\;. 
\]
\end{lem}
One can check that, at $\|x\|=1$, \eqref{eq_cap_interior} matches \eqref{eq_cap_blister}.

\begin{proof} We proceed exactly as in the proof of Lemma~\ref{l_cap_blister}. Without loss of generality 
we assume that~$x$ is on the positive semi-axis. 
We parametrize the intersection points of~$\partial\B(1)$ and~$\partial\B(x,r)$ by their angles $\pm \pi \hat \phi$ 
as seen from the origin and $\pm \pi \hat \psi$ 
as  seen from the point $x$.  
Writing the identity $e^{ i\pi \hat \phi}= x + r e^{  i\pi \hat \psi}$
 in terms of $\sin ( \pi \hat \phi/2), \sin ( \pi \hat \psi/2)$, 
we obtain
that $\hat \phi=\phi$ and  $\hat \psi=\psi$ with $\phi, \psi$ 
defined in the Lemma.
The tangent vectors $v_{1,2}$ at point $e^{ i\pi\phi}$ 
 now have directions $\pi( \phi + 1/2)$ and $\pi(\psi-1/2)$, 
so the sector has angle  $\pi (1+\phi-\psi)$. 
Applying the map $f_1$ from~\eqref{eq:f_1}, the domain
 $\C_\infty\setminus A_r$ gets mapped
to the sector determined the rays $e^{ i\pi\phi}$ 
and $e^{ -i\pi(1-\psi)}$ containing the point 1. 
Next we apply the function 
\[
z \mapsto e^{ i\pi\frac{1-\psi}{1+\phi-\psi}} z^{\frac{1}{1+\phi-\psi}}
\]
which maps this sector to the upper half-plane. 
Finally we compose with the function
 $z \mapsto \frac{z-\bar a}{z-a}$ with
 $a=e^{ i\pi\frac{1-\psi}{1+\phi-\psi}}$, i.e., the image of 1 
by the previous function. 

Collecting the pieces, we then obtain that the map 
\begin{equation}
\label{final_map2}
 f(z) =  \frac{\big(\frac{z-e^{i\pi\phi}}
{z-e^{-i\pi\phi}}\big)^{\frac{1}{1+\phi-\psi}}
-e^{-2i\pi\frac{1-\psi}{1+\phi-\psi}}}
{\big(\frac{z-e^{i\pi\phi}}
{z-e^{-i\pi\phi}}\big)^{\frac{1}{1+\phi-\psi}}-1}
\end{equation}
sends $\C_\infty\setminus A_r$ to $\C_\infty\setminus \B(1)$.

Next, we expand the
above function at $\infty$:
\[
 \Big(\frac{z-e^{i\pi\phi}}
{z-e^{-i\pi\phi}}\Big)^{\frac{1}{1+\phi-\psi}}
= 1-\frac{2i\sin(\pi\phi)}{1+\phi-\psi} z^{-1}+O(z^{-2}),
\]
yielding, as $z\to\infty$,
\begin{equation}\nn
f(z)= cz + O(1)\;, 
\qquad c= \frac{i (1+\phi-\psi)}{2 \sin (\pi \phi)} \big( 1- e^{-2i\pi\frac{1-\psi}{1+\phi-\psi}}\big)\;,
\end{equation}
and so
\begin{align*}
|c|&= \frac{(1+\phi-\psi)}{ \sin (\pi \phi)} \sin \big(  \pi \frac{1-\psi}{1+\phi-\psi}\big)\\
&= \frac{(1+\phi-\psi)}{ \sin (\pi \phi)} \sin \big(  \pi \frac{\phi}{1+\phi-\psi}\big).
\end{align*}
This gives the exact formula in~\eqref{eq_cap_interior}.
Then, the asymptotic expression as $r \to 1-\|x\|$ follows from a tedious expansion.
\end{proof}
Next, we need a formula for the capacity of a union
of three unit disks; note that, unlike the discrete 
case (see Lemma 3.8 in \cite{CPV16}),
in the continuous setting there is no closed-form
exact expression for this capacity (at least the authors
are unaware of such).


\begin{lem}
\label{l_cap_3disks}
 For $y,z$ such that the disks $\B(1), \B(y,1), \B(z,1)$ are disjoint,
abbreviate $A=\B(1)\cup \B(y,1)\cup \B(z,1)$ and
 $a=\ln\|y\|$, $b=\ln\|z\|$, $c=\ln\|y-z\|$.
Assume that, for some fixed~$\eps_0>0$, it holds that 
$\min\{a,b,c\}\geq \eps_0 \max\{a,b,c\}$.
 Then, we have
\begin{equation}
\label{eq_cap_3disks}
\capa(A) = \frac{2}{\pi} \cdot
\frac{2abc+O\big(\frac{\ln^2(a\vee b\vee c)}{a\wedge b\wedge c}
\big)}{2(ab +ac+bc)-(a^2+b^2+c^2)+O\big(\frac{\ln(a\vee b\vee c)}
{a\wedge b\wedge c}\big)}.
\end{equation}
\end{lem}

\begin{proof}
 Let $r=\lceil \|y\|\vee \|z\|\rceil + 1$, and observe that $A\subset \B(r)$.
 The idea is to use~\eqref{uncond_avoid_A}
with~$x=(r^3,0)$ (so that $\ln\|x\|=3\ln r$) and $R=r^5$.
In this situation, \eqref{uncond_avoid_A} implies that 
\begin{equation}
\label{eq_3disk_R}
 \frac{\pi}{2}\capa(A) = \Big(1-\frac{2
+ O\big(\frac{1}{r^2}\big)}{5\IP_x[\tau(A)<\tau(R)]}\Big)\ln R 
 +  O\big(\textstyle\frac{\ln r}{r^2}\big).
\end{equation}
Then, we proceed similarly to the proof of Lemma~\ref{l_escape_from_disk}.
Let 
\begin{align*}
 p_1 &= \IP_x\big[\tau(A)<\tau(R),\tau(A)=\tau(\B(1))\big],\\
 p_2 &= \IP_x\big[\tau(A)<\tau(R),\tau(A)=\tau(\B(y,1))\big],\\
 p_3 &= \IP_x\big[\tau(A)<\tau(R),\tau(A)=\tau(\B(z,1))\big],
\end{align*}
so $\IP_x[\tau(A)<\tau(R)]=p_1+p_2+p_3$.
Next, denote (recall~\eqref{nothit_r})
\begin{align*}
 h_1 &= \IP_x\big[\tau(\B(1))<\tau(R)\big] = \frac{2}{5},\\
 h_2 &= \IP_x\big[\tau(\B(y,1))<\tau(R)\big] = \frac{2}{5}
  \big(1+ O\big(\textstyle\frac{1}{r^2\ln r}\big)\big),\\
 h_3 &= \IP_x\big[\tau(\B(z,1))<\tau(R)\big]= \frac{2}{5}
  \big(1+ O\big(\textstyle\frac{1}{r^2\ln r}\big)\big),
\end{align*}
(observe that $\frac{\ln\|x\|}{\ln R} = \frac{3}{5}$).
Let $\ha=a/\ln R$, $\hb=b/\ln R$,
$\hc=c/\ln R$. Then, let us denote
\begin{align*}
 q_{12} &= \IP_{\nu_1}\big[\tau(\B(y,1))<\tau(R)\big]
             = 1-\ha 
      \big(1+ O\big(\textstyle\frac{1}{\|y\|\ln \|y\|}\big)\big),\\
 q_{13} &= \IP_{\nu_1}\big[\tau(\B(z,1))<\tau(R)\big]
             = 1-\hb 
      \big(1+ O\big(\textstyle\frac{1}{\|z\|\ln \|z\|}\big)\big),\\
 q_{21} &= \IP_{\nu_2}\big[\tau(\B(1))<\tau(R)\big]
             = 1-\ha 
      \big(1+ O\big(\textstyle\frac{1}{\|y\|\ln \|y\|}\big)\big),\\
 q_{23} &= \IP_{\nu_2}\big[\tau(\B(z,1))<\tau(R)\big]
             = 1-\hc
      \big(1+ O\big(\textstyle\frac{1}{\|y-z\|\ln \|y-z\|}\big)\big),\\
 q_{31} &= \IP_{\nu_3}\big[\tau(\B(1))<\tau(R)\big]
             = 1-\hb
      \big(1+ O\big(\textstyle\frac{1}{\|z\|\ln \|z\|}\big)\big),\\
 q_{32} &= \IP_{\nu_1}\big[\tau(\B(y,1))<\tau(R)\big]
             = 1-\hc
      \big(1+ O\big(\textstyle\frac{1}{\|y-z\|\ln \|y-z\|}\big)\big),
\end{align*}
where~$\nu_1$ (respectively, $\nu_2$ and $\nu_3$) is the entrance
measure to $\B(1)$ (respectively, to $\B(y,1)$ and $\B(z,1)$)
conditioned on $\{\tau(1)<\tau(\B(y,1))\wedge
\tau(\B(z,1))\wedge \tau(R)\}$ (respectively,
on $\{\tau(\B(y,1))<\tau(1)\wedge \tau(\B(z,1))\wedge \tau(R)\}$ 
and $\{\tau(\B(z,1))<\tau(1)\wedge \tau(\B(y,1))\wedge \tau(R)\}$).

It is elementary to see that, as a general fact,
\begin{align*}
h_1 &= p_1 + p_2 q_{21} + p_3q_{31},\\
h_2 & = p_1 q_{12} + p_2 + p_3q_{32},\\
h_3 & = p_1 q_{13} + p_2q_{23} + p_3.
\end{align*}
Next, we solve the above system of linear equations with respect 
to $p_{1,2,3}$ and then use~\eqref{eq_3disk_R}
to obtain~\eqref{eq_cap_3disks}.
 We omit the precise calculations which are elementary but long; however, 
to convince the reader that the answer is correct, 
let us forget for a moment about the $O$'s in the above
expressions for $h$'s and $q$'s and see where 
it will lead us. 
Denote by~$\textbf{1}$ the column vector 
with all the coordinates being equal to~$1$, and define
the matrix
\[
 L= 
      \begin{pmatrix}
      1 & 1-\ha &1-\hb\\
      1-\ha & 1 & 1-\hc\\
      1-\hb & 1-\hc& 1     
     \end{pmatrix}.
\]
Let $p'_{1,2,3}$ be the solutions of the above equations
``without the $O$'s'', i.e., with $\frac{2}{5}$ on the place
of $h_{1,2,3}$, and $1-\ha$ (respectively, $1-\hb$ and $1-\hc$)
on the place of~$q_{12}$ (respectively, $q_{13}$ and~$q_{23}$).
Clearly, we have that 
$p'_1+p'_2+p'_3 = \mathbf{1}^{\mathsf T} L^{-1}\mathbf{1}$.
Then, the right-hand side of~\eqref{eq_3disk_R}
would become 
\begin{align*}
 \Big(1-\frac{1}{\mathbf{1}^{\mathsf T} L^{-1}\mathbf{1}}\Big)\ln R
  &= \Big(1- \frac{2(\ha\hb+\ha\hc+\hb\hc)-(\ha^2+\hb^2+\hc^2)
  -2\ha\hb\hc}{2(\ha\hb+\ha\hc+\hb\hc)-(\ha^2+\hb^2+\hc^2)}\Big)\ln R\\
 &=\frac{2\ha\hb\hc}{2(\ha\hb+\ha\hc+\hb\hc)-(\ha^2+\hb^2+\hc^2)}\ln R\\
&= \frac{2abc}{2(ab+ac+bc)-(a^2+b^2+c^2)},
\end{align*}
which indeed agrees with~\eqref{eq_cap_3disks}.
As a final remark,
let us also observe that inserting the $O$'s back is not 
problematic, because an elementary calculation shows that 
$|\det L|=2(\ha\hb+\ha\hc+\hb\hc)-(\ha^2+\hb^2+\hc^2)-2\ha\hb\hc$
is bounded away from~$0$ (note that, due to the triangle
inequality, at least \emph{two} of the quantities $a,b,c$
should be approximately equal to~$\ln r$,
and so at least two of $\ha,\hb,\hc$
are approximately~$\frac{1}{5}$; also, we assumed 
that the smallest of them must be bounded away from~$0$).
\end{proof}


We also need to compare the harmonic measure on a set 
(typically distant from the origin) to
the entrance measure of~$\hW$ started far away from that set.
\begin{lem}
\label{l_entrance_hat_s}
Assume that the compact subset $A$ of $\R^2$ and
$x\notin \B(1)$ are such  that
$2 \, \diam(A)+1<\min\!\big(\!\dist(x,A), 
\,\frac{1}{4} \dist(\B(1),A)\big)$.
Abbreviate $u=\diam(A)$, $s=\dist(x,A)$. 
Assume also that $A'\subset \R^2$ is
compact or empty,
and such that $\dist(A,A')\geq s+1$ 
(for definiteness, we adopt the convention $\dist(A,\emptyset)=\infty$ for any~$A$).
Then, for all $M \subset \partial A$, it holds that
\begin{equation}
\label{eq_entrance_hat_s}
\IP_x\big[\hW_{\htau(A)}\in M\mid \htau(A)<\infty,
   \htau(A)<\htau(A')\big] 
= \hm_A(M)\Big(1+O\Big(\frac{u}{s}\Big)\Big).
\end{equation}
\end{lem}

\begin{proof}
This result is analogous to Lemma~3.10 of~\cite{CPV16},
but, since the proof of the latter contains an inaccuracy,
we give the proof here, also noting that the same argument
works in the discrete setting as well. 

Let $z_0\in A$ be such that $\|z_0-x\|=s$, and observe that
the assumptions imply that 
$(\B(1)\cup A')\cap \B(z_0,s) = \emptyset$.
Let us fix some~$R$ such that $A\cup A' \cup\{x\} \subset \B(R)$,
and abbreviate
\[
 G = \big\{\tau(A)<\tau(A')\wedge \tau(R), \tau(R)<\tau(1)\big\}.
\]
Define the (possibly infinite) random variable
\[
 \sigma = 
  \begin{cases}
   \inf\big\{t\geq 0: W_t\in\B(z_0,{\textstyle\frac{s}{2}}),
   W_{[t,\tau(A)]}\cap \partial\B(z_0,s) = \emptyset\big\} &
    \text{ on } G,\\
    \infty & \text{ on } G^\complement
  \end{cases}
\]
to be the moment when the \emph{last} excursion between
$\B(z_0,\frac{s}{2})$ and~$\partial\B(z_0,s)$ starts;
formally, we also set $W_\sigma=\infty$ on $\{\sigma=\infty\}$.
Let us stress that~$\sigma$ is \emph{not} a stopping time;
it was defined in such a way that the law of the excursion
that begins at time~$\sigma$ is the law of a Brownian excursion
\emph{conditioned} to reach~$A$ (and then~$\partial B(R)$)
before going back to~$\B(1)$.
Let~$\nu$ be the joint law of 
the pair $(\sigma,W_\sigma)$ under~$\IP_x$.
Abbreviate also $\mathcal{H}=\R_+\times\partial\B(z_0,\frac{s}{2})$
and observe that $\int_{\mathcal{H}}\,d\nu(t,y) = \IP_x[G]$.

Now, using Lemma~\ref{l_conn_W_hW}, \eqref{hitting_BM},
 and~\eqref{RN_entrance},
we write
\begin{align*}
\lefteqn{
 \IP_x\big[\hW_{\htau(A)}\in M, \htau(A)<\htau(R)\wedge\htau(A')\big]
}\\ 
&= \IP_x \big[W_{\tau(A)}\in M, \tau(A)<\tau(R)\wedge\tau(A')
  \mid  \tau(R)<\tau(1)\big]\\
&= \frac{\ln R}{\ln\|x\|} \IP_x \big[W_{\tau(A)}\in M,
  \tau(A)<\tau(A')\wedge \tau(R), \tau(R)<\tau(1) \big]\\
&=\frac{\ln R}{\ln\|x\|} 
\int_{\mathcal{H}}
\IP_y \big[W_{\tau(A)}\in M\mid \tau(A)
  <\tau\big(\partial\B(z_0,{\textstyle\frac{s}{2}})\big) \big]
\,d\nu(t,y)\\
&=\frac{\ln R}{\ln\|x\|} \hm_A(M)\Big(1+O\Big(\frac{u}{s}\Big)\Big)
  \IP_x\big[\tau(A)<\tau(A')\wedge \tau(R), \tau(R)<\tau(1)\big]\\
&=\hm_A(M)\Big(1+O\Big(\frac{u}{s}\Big)\Big)
  \IP_x\big[\tau(A)<\tau(A')\wedge \tau(R)\mid \tau(R)<\tau(1)\big]\\
&=\hm_A(M)\Big(1+O\Big(\frac{u}{s}\Big)\Big)
  \IP_x\big[\htau(A)<\htau(A')\wedge \htau(R)\big],
\end{align*}
and we conclude the proof of Lemma~\ref{l_entrance_hat_s}
by sending~$R$ to infinity.
\end{proof}

\subsection{Excursions}
\label{s_excursions}
In this section we define \emph{excursions} of the Brownian
random interlacements and the Brownian motion on the torus.
The corresponding definitions for the discrete random interlacements
and the simple random walk on the torus are contained
in Section~3.4 of~\cite{CPV16}; since the definitions
are completely analogous in the continuous case, we make this
section rather sketchy.

Consider two closed sets $A$ and~$A'$
such that $A\subset (A')^o \subset \R^2_n$
(usually, they will be disks),
and let~$T_n(M)$ be the hitting time of a set~$M\subset \R^2_n$ 
for the process~$X$.
 By definition, an excursion~$\vr$ is a continuous
(in fact, Brownian) path that starts at~$\partial A$ and ends
on its first visit to~$\partial A'$, i.e., $\vr=(\vr_t, t\in [0,v])$,
where $\vr_0\in\partial A$, $\vr_v\in\partial A'$,
$\vr_s\notin\partial A'$ for all $s\in[0,v)$. 
To define these excursions, consider 
the following sequence of stopping times: 
\begin{align*}
 D_0 &= T_n(\partial A'),\\
 J_1 &= \inf\{t> D_0 : X_t \in \partial A\},\\
 D_1 &= \inf\{t> J_1 : X_t \in \partial A'\},
\end{align*}
and
\begin{align*}
 J_k &= \inf\{t> D_{k-1} : X_t \in \partial A\},\\
 D_k &= \inf\{t> J_k : X_t \in \partial A'\},
\end{align*}
for $k\geq 2$. Then,
denote by $Z^{(i)}=X_{|[J_i, D_i]}$ the $i$th excursion
of~$X$ between~$\partial A$ and~$\partial A'$, for $i\geq 1$.
Also, let $Z^{(0)}=X_{|[0,D_0]}$ be the ``initial''
excursion (it is possible, in fact, that it does not intersect
the set~$A$ at all). 
Recall that $t_\alpha:=\frac{2\alpha}{\pi}n^2\ln^2 n$
and define
\begin{align}
 N_\alpha &= \max\{k : J_k\leq t_\alpha\},\label{df_Na}
\\
 N'_\alpha &= \max\{k : D_k\leq t_\alpha\}\label{df_N'a}
\end{align}
to be the number of incomplete (respectively, complete)
excursions up to time~$t_\alpha$.

Observe that, quite analogously to the above, we can define
the excursions of the conditioned diffusion~$\hW$ 
between~$\partial A$ and~$\partial A'$
(this time, $A$ and~$A'$ are subsets of~$\R^2$);
 since~$\hW$ is transient,
the number of those will be a.s.\ finite. 
Next, we also define the excursions of (Brownian) random
interlacements. Suppose that the trajectories 
of the $\hW$-diffusions that intersect~$A$ are enumerated according
to the points of the generating one-dimensional Poisson process
(cf.\ Definition~\ref{df_BRI}).
For each trajectory from that list (say, the $j$th one,
denoted~$\hW^{(j)}$ and time-shifted in such a way 
that $\hW^{(j)}_s\notin A$ for all $s<0$ and $\hW^{(j)}_0 \in A$)
define the stopping times 
\begin{align*}
 {\hat J}_1 &= 0,\\
 {\hat D}_1 &= \inf\big\{t> {\hat J}_1 : \hW^{(j)}_t \in \partial A'\big\},
\end{align*}
and
\begin{align*}
 {\hat J}_k &= \inf\big\{t> {\hat D}_{k-1} : \hW^{(j)}_t \in \partial A\big\},\\
 {\hat D}_k &= \inf\big\{t> {\hat J}_k : \hW^{(j)}_t \in \partial A'\big\},
\end{align*}
for $k\geq 2$. Let~$\ell_j=\inf\{k:{\hat J}_k=\infty\}-1$ be the
number of excursions corresponding to the $j$th trajectory.
The excursions of $\BRI(\alpha)$ between~$\partial A$ 
and~$\partial A'$ are defined by
\[
 \hZ ^{(i)}=\hW^{(j)}_{[{\hat J}_m, {\hat D}_m]},
\]
where 
$ i = m+\sum_{k=1}^{j-1} \ell_k$,
and $m=1,2,\ldots, \ell_j$. Let $R_\alpha$ be the number of
trajectories intersecting~$A$ on level~$\alpha$,
and denote ${\hat N}_\alpha = \sum_{k=1}^{R_\alpha} \ell_k$
to be the total number of excursions of~$\BRI(\alpha)$ 
between~$\partial A$ and~$\partial A'$.

Observe also that the above construction makes sense with $\alpha=\infty$
as well; we then obtain an infinite sequence of excursions
of $\BRI$ (=$\BRI(\infty)$) between~$\partial A$
and~$\partial A'$.

Next, we need to control the number of excursions between
the boundaries of two concentric disks on the torus:
\begin{lem}
\label{l_excursions_torus}
Consider the random variables $(J_k,D_k)$ defined in this section with 
$A=B(r)$ and $A'=B(R)$. Assume that $1<r<R<\frac{n}{2}$,
$m\geq 2$, and $\delta\in(0,c_0)$ for some $c_0>0$.
Then, there exist positive constants $c_1,c_2$ such that 
\begin{equation}
\label{eq_excursions_torus_g}
\IP\Big[J_m\notin\Big(\frac{(1 \!-\!\delta)m}{\pi}n^2 \ln \frac{R}{r},
\frac{(1\!+\!\delta)m}{\pi}n^2 \ln \frac{R}{r}\Big)\Big] 
    \leq   c_1\exp\Big(-c_2\delta^2 m
    \frac{R\big(1 \!- \!\frac{r}{R}\big)^6}{n\ln^2 r^{-1}}\Big),
\end{equation}
and the same result holds with~$D_m$ on the place
of~$J_m$.
\end{lem}

\begin{proof}
 This is Proposition~8.10 of~\cite{BK14}, with small adaptations,
 since we are working with torus of size~$n$ 
in the continuous setting as well.
\end{proof}

\section{Proofs of the main results}
\label{s_proofs}

\subsection{Proof of Theorems~\ref{p_indeed_BRI}, 
\ref{t_basic_prop} and \ref{t_basic_size}}

\begin{proof}[Proof of Proposition~\ref{p_indeed_BRI}]
 We start with a preliminary observation. 
Consider the process~$\hW$ started at some~$x\in\R^2$
with $\|x\|=r>1$, and consider the random variable
$H=\inf_{t>0}\|\hW_t\|$ to be the minimal distance
of the trajectory to the origin; note that~$H>1$ a.s..
By~\eqref{escape_condBM}, it holds that $\IP_x[H\leq s] 
= \frac{\ln s}{\ln r}$, so~$H$ has density 
$f(s)=\frac{1}{s\ln r}\1{s\in[1,r]}$. But then, 
using Lemma~\ref{l_W_go_in}, 
we
see that the trace of $\BRI(\alpha)$ on~$\B(r)$
can be obtained in the following way: take 
$N\sim \text{Poisson}(2\alpha\ln r)$
 particles and place them on~$\partial \B(r)$ uniformly
and independently; then let these particles perform independent
$\hW$-diffusions. Indeed, the trace left by these diffusions 
on~$\B(r)$ has the same law as the trace of $\BRI(\alpha)$
defined as in Definition~\ref{df_BRI}.

Now, we are ready for the proof of part (i). 
Using~\eqref{avoid_A_forever}
and recalling that $\IE s^{N} = e^{2\alpha(s-1)\ln r}$, we write
\begin{align*}
 \IP\big[A\cap\BRI(\alpha)=\emptyset\big] &=
 \IE\big(\IP\big[A\cap\BRI(\alpha)=\emptyset\mid N\big]\big)\\
 &= \IE\Big( 1-\frac{\pi}{2}\big(1+O(r^{-1})\big)
 \frac{\capa(A)}{\ln r}\Big)^N\\
 &= \exp\big(-\pi\alpha\capa(A)
   (1+O(r^{-1}))\big),
\end{align*}
and we obtain~\eqref{eq_vacant_Bro} by sending~$r$ to infinity.

Observe also, since $\capa(\B(r))=\frac{2}{\pi}\ln r$
by~\eqref{capa_disk}, the above 
construction of $\BRI(\alpha)$ on~$\B(r)$
agrees with the ``constructive description''
in the part~(ii) of Proposition~\ref{p_indeed_BRI}
(note that $2\ln r = \pi \capa(\B(r))$). In fact, a calculation
completely analogous to the above 
(i.e., fix~$A$, start with independent
particles on~$\partial\B(r)$, and then send~$r$ to infinity)
provides the proof of the part~(ii).
\end{proof}

As we mentioned in Section~\ref{s_results}, 
Theorems~\ref{t_basic_prop}--\ref{t_torus} are quite analogous
to the corresponding results of~\cite{CP16,CPV16} for the 
discrete two-dimensional random interlacements,
and their proofs are quite analogous as well.
Therefore, we give only a sketch of the proofs, 
since the adaptations 
to the continuous setting
are usually quite straightforward.

\begin{proof}[Proof of Theorem~\ref{t_basic_prop}]
The proof of part~(i) follows from the invariance of the capacity
with respect to isometries of~$\R^2$.
Using~\eqref{eq_cap_distdisk}, we obtain that
\[
\capa(\B(1)\cup\B(x,1)) = \frac{1}{\pi}
\big(1+O\big((\|x\|\ln \|x\|)^{-1}\big)\big)\ln \|x\|,
\]
and, together with~\eqref{eq_vacant_Bro}, this implies the part~(ii)
(the more general formula~\eqref{properties_BRI_ii'} 
follows from~\eqref{eq_cap_closedisk}).
Next,
observe that, by symmetry, Theorem~\ref{t_basic_prop}~(ii),
 and~\eqref{eq_cap_distantset} imply that
\begin{align*}
\IP[A\subset\V^\alpha \mid x\in\D_1(\V^\alpha)] 
&= \exp\Big(-\pi\alpha\big(\capa(A\cup\B(x,1))
-\capa(\B(1)\cup\B(x,1))\big)\Big)\\
&=\exp\Bigg(\!-\pi\alpha
\Big(
\frac{2}{\pi}\cdot
  \frac{\ln^2\|x\| +O(\|x\|^{-1}(r+1)\ln\|x\|\ln(r+1))}
{2\ln\|x\|-\frac{\pi}{2}\capa(A)+O(\|x\|^{-1}(r+1))}\\
& \qquad\qquad\qquad\qquad\qquad -\frac{1}{\pi}
\big(1+O\big((\|x\|\ln \|x\|)^{-1}\big)\big)\ln \|x\|\Big)\!\Bigg)\\
&=\exp\Bigg(-\frac{\pi\alpha}{4}\capa(A)
\frac{1+O\big(\frac{r\ln r }{\|x\|}\big)}
{1-\frac{\pi\capa(A)}{4\ln\|x\|}
+O\big(\frac{r\ln r}{\|x\|\ln \|x\|}\big)}
\Bigg),
\end{align*}
thus proving the part~(iii). Finally, the part~(iv)
follows from Lemma~\ref{l_cap_3disks} and~\eqref{properties_BRI_ii}.
\end{proof}

\begin{proof}[Proof of Theorem~\ref{t_basic_size}]

 The part~(i) follows directly from~\eqref{properties_BRI_ii}.

Let us deal with the part~(ii). 
First, we explain how the fact that $\D_s(\V^\alpha)$ is 
a.s.\ bounded for $\alpha>1$ implies the second part of the statement.
For a fixed~$\alpha$, we first choose $\eps>0$ such that
$\alpha-\eps>1$, and then use the superposition property~\eqref{superp_BRI}: 
$\D_s(\V^{\alpha-\eps})$
is a.s.\ compact, and with positive probability 
the ``$\BRI$-sausages'' $\BRI(\eps)+\B(s)$
will cover $\D_s(\V^{\alpha-\eps})\setminus \B(1-s+\delta)$.
The same kind of argument works for proving that this 
probability tends to~$1$ as $\alpha\to\infty$: for any~$\eps'>0$
there is~$R$ such that 
$\IP[\D_s(\V^{\alpha})\setminus \B(R)=\emptyset]>1-\eps'$;
then, we have many tries to cover~$\B(R)\setminus\B(1-s+\delta)$ by independent
copies of~$\BRI(1)$. 

Note that in~\cite{CPV16} one does not need the FKG
inequality in the proof of the corresponding statement,
due to the same kind of argument.

 From this point to the end of the proof, 
we consider the case $s=1$ to simplify the notations. 
The general case is similar.
To complete the proof of part~(ii), 
it remains to show that~$\D_1(\V^\alpha)$
is a.s.\ bounded for any~$\alpha>1$.
Let us abbreviate $a_0:=(1+\sqrt{2})^{-1}$, and consider
the square grid $2a_0\Z^2 \subset \R^2$. It is elementary 
to obtain that for any $x\in\R^2$ there exists $y\in 2a_0\Z^2$
such that $\B(y,a_0)\subset\B(x,1)$. This means that
\begin{equation}
\label{discretization}
 \{y\in 2a_0\Z^2 : \B(y,a_0)\subset\V^\alpha\} \text{ is finite }
\; \Rightarrow \;
\D_1(\V^\alpha)\text{ is bounded}.
\end{equation}

Let~$r>8$ (so that $\ln r > 2$).
 We use Lemmas~\ref{l_conn_W_hW} and~\ref{l_exit_disks} 
to obtain that, for any $x\in\partial\B(2r)$ and $y \in \B(r) \ \B(r/2)$,
\begin{align}
\lefteqn{ 
\IP_x\big[\htau(\B(y,a_0))<\htau(\partial\B(r\ln r))\big]
} \nonumber\\
 &=
  \frac{\ln(r\ln r)}{\ln (2r)}
\IP_x\big[\tau(\B(y,a_0))<\tau(\partial\B(r\ln r))<\tau(\B(1))\big]
\nonumber\\
 &= (1+o(1))\big(\IP_x\big[\tau(\B(y,a_0))<\tau(\partial\B(r\ln r))\big]
\nonumber\\
& \qquad\qquad - \IP_x\big[\tau(\B(y,a_0))
<\tau(\partial\B(r\ln r)),
\tau(\B(1))<\tau(\partial\B(r\ln r))\big]\big)
\label{f63interm}\\
 &= \frac{\ln\ln r}{\ln r}(1+o(1))
\label{f63}
\end{align}
(note that the first probability in~\eqref{f63interm}
is $\frac{\ln\ln r}{\ln r}(1+o(1))$ by Lemma~\ref{l_exit_disks}
and it is straightforward to obtain that the second
one is $O(\frac{(\ln\ln r)^2}{\ln^2 r})$).

Let~$N_{\alpha,r}$ be the number of $\hW$-excursions of $\BRI(\alpha)$
between~$\partial B(2r)$ and $\partial B(r\ln r)$. 
By~\eqref{escape_condBM} and~\eqref{capa_disk},
$N_{\alpha,r}$ is a compound Poisson random variable with
rate $2\alpha\ln(2r)$ and with 
Geometric($1-\frac{\ln 2r}{\ln(r\ln r)}$) terms.
Analogously to~(66) of~\cite{CPV16}, we can show that
\begin{equation}
\label{f66}
\IP\Big[N_{\alpha,r} \leq b\frac{2\alpha\ln^2 r}{\ln\ln r}\Big] 
\leq r^{-2\alpha(1-\sqrt{b})^2 (1+o(1))}
\end{equation}
for $b<1$.
Now, \eqref{f63} implies that for $y\in \B(r)\setminus \B(r/2)$
\begin{align*}
 \IP\Big[\B(y,a_0)\text{ is untouched by first }
 b\frac{2\alpha\ln^2 r}{\ln\ln r}\text{ excursions}\Big] 
& \leq \Big(1-\frac{\ln\ln r}{\ln r}
(1+o(1))\Big)^{b\frac{2\alpha\ln^2 r}{\ln\ln r}}
\\
& = r^{-2b\alpha(1+o(1))},
\end{align*}
so, by the union bound,
\begin{align}
\lefteqn{
 \IP\Big[\text{there exists }y\in 2a_0\Z^2\cap(\B(r) \!\setminus  \! \B(r/2))
\text{ such that }
\B(y,a_0)\in\V^\alpha , N_{\alpha,r} >
 b\frac{2\alpha\ln^2 r}{\ln\ln r}
\Big] 
} ~~~~~~~~~~~~~~~~~~~~~~~~~~~~~~~~~~~~~~~~~~~~~~~~~~~~~~~~~~~~~~~~~~~
~~~~~~~~~~~~~~~~~
\nonumber\\
&\leq r^{-2(b\alpha-1)(1+o(1))}.\phantom{\sum^A}
\label{cover_ring_b}
\end{align}
Using~\eqref{f66} and~\eqref{cover_ring_b} with 
$b=\frac{1}{4}\big(1+\frac{1}{\alpha}\big)^2$, we
obtain that
\begin{equation}
\label{covered_ring}
 \IP\big[\text{there exists }y\in 2a_0\Z^2\cap(\B(r) \!\setminus  \! \B(r/2))
\text{ such that }\B(y,a_0)\in\V^\alpha\big] 
\leq r^{-\frac{\alpha}{2}(1-\frac{1}{\alpha})^2(1+o(1))}.
\end{equation}
This implies that the set~$\D_1(\V^\alpha)$ is a.s.\
bounded, since
\begin{equation*}
 \{\D_1(\V^\alpha)\text{  unbounded}\}
= \big\{\D_1(\V^\alpha)\cap\big(\B(2^n) \! \setminus  \! \B(2^{n-1})\big)
\neq \emptyset
\text{ for infinitely many }n\big\},
\end{equation*}
and the Borel-Cantelli lemma together with~\eqref{covered_ring}
imply that the probability of the latter event equals~$0$.
This concludes the proof of part~(ii) of Theorem~\ref{t_basic_size}.
\medskip

Let us now prove the {\it part~(iii)}. First,
we deal with the critical case $\alpha=1$. 
Again, the proof is essentially
the same as in~\cite{CP16}, so we present only a sketch.
For $k\geq 1$ we denote $b_k= \exp\big(\exp(3^k)\big)$, 
and let $v_k = b_ke_1 \in \R^2$.
Fix some $\gamma\in (1,\sqrt{\pi/2})$, and consider the disks
 $B_k=\B(v_k,b_k^{1/2})$ and $B'_k=\B(v_k,\gamma b_k^{1/2})$.
 Let~$N_k$ be the number of excursions
between $\partial B_k$ and~$\partial B'_k$ in $\RI(1)$.
The main idea is that, although ``in average'' the number
of those excursions will be enough to cover~$B_k$
(this is due to the fact that the expected cover time
has a \emph{negative} second-order correction, see~\cite{BK14}), 
 the fluctuations of~$N_k$ are of much bigger order than those
of the excursion counts on the torus. Therefore, $N_k$'s
will be atypically low for \emph{some} $k$'s, thus leading
to non-covering of corresponding $B_k$'s.  

Let us now present some details. 
Lemma~\ref{l_cap_distantdisk}~(i) together with
Lemma~\ref{l_escape_from_disk}~(i) imply
that~$N_k$ is (approximately) compound Poisson
with rate $\frac{4}{3\pi}\big(1+O((\ln^{-1}b_k)\big)\ln b_k$ and
Geometric$\big(\frac{2\ln\gamma}{3\ln b_k}(1+O((\ln^{-1}b_k))\big)$
terms.
Then, standard arguments (see~(72) of~\cite{CP16}) imply that
\begin{equation}
\label{Nk_Normal}
\frac{\ln\gamma}{\sqrt{6}\ln^{3/2}b_k}
\Big(N_k - \frac{2}{\ln\gamma}\ln^2 b_k\Big)
 \convlaw \text{Normal(0,1).}
\end{equation}
Observe that $\frac{\pi}{4\gamma^2}>\frac{1}{2}$ by our 
choice of~$\gamma$. Choose some~$\beta\in (0,\frac{1}{2})$
in such a way that $\beta + \frac{\pi}{4\gamma^2}> 1$, 
and define~$q_\beta>0$ to be such that 
\[
 \int_{-\infty}^{-q_\beta}\frac{1}{\sqrt{2\pi}}e^{-x^2/2}\, dx = \beta. 
\]
Consider the sequence of events
\begin{equation}
\label{df_Phi_k}
 \Phi_k = \Big\{N_k\leq \frac{2}{\ln\gamma}\ln^2 b_k - 
    q_\beta\frac{\sqrt{6}\ln^{3/2}b_k}{\ln\gamma}\Big\}.
\end{equation}
Observe that~\eqref{Nk_Normal} clearly implies that
$\IP[\Phi_k]\to\beta$ as $k\to\infty$.
Analogously to~\cite{CP16} (see the proof of~(76) there)
it is possible to obtain that
\begin{equation}
\label{cond_Ak}
 \lim_{k\to\infty}\IP[\Phi_k\mid \D_{k-1}] = \beta \qquad \text{a.s.},
\end{equation}
where~$\D_j$ is the partition generated by the 
events $\Phi_1,\ldots, \Phi_j$. 
Roughly speaking, the idea is that the sequence~$(b_k)$ grows
so rapidly, that what happens on $B'_1,\ldots,B'_{k-1}$
has almost no influence on what is seen on~$B_k$.
Using~\eqref{cond_Ak}, we then obtain
\begin{equation}
\label{mnogo_Phi_k}
 \liminf_{n\to\infty} \frac{1}{n}\sum_{j=1}^n \1{\Phi_j}
   \geq \beta \qquad \text{a.s.}
\end{equation}

Now, let~$(\hZ^{(j),k},j\geq 1)$ be 
the $\RI$'s excursions between~$\partial B_k$ 
and~$\partial B'_k$, $k\geq 1$,
constructed as in Section~\ref{s_excursions}.
Let~$(\tZ^{(j),k},j\geq 1)$ be sequences 
of i.i.d.\ excursions, with starting points chosen 
uniformly on~$\partial B'_k$.
Next, let us define the sequence of \emph{independent} events
\begin{equation}
\label{df_I_k}
 \JJ_k= \Big\{ \exists x\in B_k :
x\notin \tZ^{(j),k}\;, \quad  \forall  j\leq
\frac{2}{\ln\gamma}\ln^2 b_k - \ln^{11/9}b_k\Big\},
\end{equation}
that is, $\JJ_k$ is the event that the set~$B_k$
is not completely covered by the first 
$\frac{2}{\ln\gamma}\ln^2 b_k - \ln^{11/9}b_k$
independent excursions.

Next, fix~$\delta_0>0$ such that 
$\beta+\frac{\pi}{4\gamma^2}>1+\delta_0$.
Quite analogously to Lemma~3.2 of~\cite{CP16},
one can prove the following fact:
 For all large enough~$k$ it holds that
\begin{equation}
\label{eq_compare_RW_RI}
  \IP[\JJ_k]\geq \frac{\pi}{4\gamma^2}-\delta_0.
\end{equation}
We only outline the proof of~\eqref{eq_compare_RW_RI}: 
\begin{itemize}
 \item consider a Brownian motion on a torus
of slightly bigger size (specifically, $(\gamma+\eps_1)b_k^{1/2}$),
so that the set~$B'_k$ would ``completely fit'' there;
 \item we recall a known result (of~\cite{BK14}) that, up to time
\[
\hspace{-5ex}
 t_k = \frac{2}{\pi} 
 (\gamma \! + \! \eps_1)^2b_k\ln^2 \big((\gamma \! + \! \eps_1)b_k^{1/2}\big)
   - {\hat c} (\gamma \! + \! \eps_1)^2b_k
\ln \big((\gamma \! + \! \eps_1)b_k^{1/2}\big) 
\ln\ln \big((\gamma \! + \! \eps_1)b_k^{1/2}\big)
\]
the torus is not completely covered
with high probability;
 \item using soft local times, we \emph{couple} 
the i.i.d.\ excursions between~$\partial B_k$ 
and~$\partial B'_k$ with the 
Brownian motion's excursions between the corresponding
sets on the torus;
 \item using Lemma~2.9 of~\cite{CP16} adapted to the 
present setting (see also Section~6 of~\cite{BGP16}), we 
conclude that the set 
of Brownian motion's excursions un the torus up to time~$t_k$
is likely to contain
the set of i.i.d.\ excursions;
 \item finally, we note that the Brownian motion's excursions
will not completely cover the smaller disk with at least constant
probability, and this implies~\eqref{eq_compare_RW_RI}.
\end{itemize}

Then, analogously to (88)--(91) of~\cite{CP16}, 
we can prove that, for all but a finite number of~$k$'s,
 the set of $\frac{2}{\ln\gamma}\ln^2 b_k 
- q_\beta\frac{\sqrt{6}\ln^{3/2}b_k}{\ln\gamma}$
$\BRI$'s excursions between~$\partial B_k$ and~$\partial B'_k$ 
(recall~\eqref{df_Phi_k}) is contained in the set of 
$\frac{2}{\ln\gamma}\ln^2 b_k - \ln^{11/9}b_k$
independent excursions. 
Since (recall~\eqref{mnogo_Phi_k} and~\eqref{eq_compare_RW_RI})
$\beta+\frac{\pi}{4\gamma^2}-\delta_0>1$, for at least 
a positive proportion of $k$'s the events~$\Phi_k\cap\JJ_k$
occur. This implies that $\D_1(\V^1)\cap B_k \neq \emptyset$
for infinitely many $k$'s, thus proving 
that~$\D_1(\V^1)$ is a.s.\ unbounded.
\medskip

Now, it remains only to prove that~\eqref{eq_emptydisk}
holds for $\alpha<1$. 
Fix some $\gamma\in (1,\sqrt{\pi/2})$ and $\beta\in (0,1)$, 
which will be later taken close  to~$1$, and fix some set of
 non-intersecting 
disks~${\tilde B}'_1=\B(x_1,\gamma r^\beta),
\ldots,{\tilde B}'_{k_r}=\B(x_{k_r},\gamma r^\beta)\subset
\B(r)\setminus \B(r/2)$, with cardinality $k_r = \frac{1}{8}r^{2(1-\beta)}$.
Denote also ${\tilde B}_j:= \B(x_j, r^\beta)$,
$j=1,\ldots,k_r$. 

By Lemma~\ref{l_cap_distantdisk}~(i), 
 the number of $\hW$-diffusions in $\BRI(\alpha)$ 
intersecting a given disk~${\tilde B}_{j}$
has Poisson law with parameter 
$\lambda=(1+o(1))\frac{2\alpha}{2-\beta}\ln r$. 
By Lemma~\ref{l_escape_from_disk}~(i), the probability that
a $\hW$-diffusion started from any $y\in\partial {\tilde B}'_{j}$
does not hit~${\tilde B}_{j}$ is 
$(1+o(1))\frac{\ln\gamma}{(2-\beta)\ln r}$. 
Let ${\hat N}_\alpha^{(j)}$ be the total number of excursions 
between~$\partial {\tilde B}_{j}$ and~$\partial {\tilde B}'_{j}$ 
in $\BRI(\alpha)$.
Quite analogously to~(57) of~\cite{CPV16}, we obtain
\begin{align}
\IP\Big[{\hat N}_\alpha^{(j)}\geq b\frac{2\alpha\ln^2 r}
{\ln\gamma}\Big] &\leq
\exp\Big(-(1+o(1))\big(\sqrt{b}-1\big)^2 \frac{2\alpha}{2-\beta}
\ln r\Big) \nonumber\\
&=  r^{-(1+o(1))(\sqrt{b}-1)^2 \frac{2\alpha}{2-\beta}}.
\label{LD_numb_exc}
\end{align}
Let~$U_b$ be the set 
\[
 U_b = \Big\{j\leq k_r: {\hat N}_\alpha^{(j)} 
< b\frac{2\alpha\ln^2 r}{\ln\gamma}\Big\}.
\]
Then, just as in~(58) of~\cite{CPV16}, we obtain that 
\begin{equation}
\label{manydisks}
\IP\big[|U_b|\geq k_r/2\big] \geq 1-2r^{-(1+o(1))(\sqrt{b}-1)^2
\frac{2\alpha}{2-\beta}}.
\end{equation}

We then again use the idea of 
 comparing the (almost) Brownian excursions between
$\partial {\tilde B}_{j}$ and~$\partial {\tilde B}'_{j}$
 with the Brownian 
excursions on a (slightly larger) torus containing
a copy of~${\tilde B}'_{j}$.
In this way, we see that the ``critical''
number of excursions there is $\frac{2\beta^2 \ln^r}{\ln \gamma}$,
up to terms of smaller order. 
So, let us assume that $\beta<1$ 
is such that $\beta^2<\alpha$.

We then  repeat the arguments we used in the case $\alpha=1$
(that is, use soft local times for constructing 
the independent excursions together with the Brownian
motion's excursions etc.) to prove that the probability
that all the disks $({\tilde B}_{j}, j=1,\ldots,k_r)$
are completely covered is small (in fact, of a subpolynomial order
in~$r$), to show that, for any fixed $h>0$
\[
  \IP\big[\D_1(\V^\alpha)\cap \big(B(r)\setminus B(r/2)\big)
=\emptyset\big]
   \leq 2r^{-(1+o(1))(\sqrt{b}-1)^2
\frac{2\alpha}{2-\beta}} + o(r^{-h})
\]
as $r\to \infty$.
Since $b\in (1,\alpha^{-1})$ can be arbitrarily close
to~$\alpha^{-1}$ and $\beta\in (0,1)$ 
can be arbitrarily close to~$1$,
this concludes the proof of (\ref{eq_emptydisk}).
\end{proof}

\subsection{
Proofs for the cover process on the torus}

\begin{proof}[Proof of Theorem~\ref{t_torus}]
The proof of this fact parallels that of 
Theorem~2.6 of~\cite{CPV16}, with some evident
adaptations. Therefore, in the following
we only recall the main steps of the argument.
With $T_n(M)$ the hitting time of a set~$M\subset \R^2_n$ 
by the process~$X$,  
 denote for~$x\in\R^2_n$, 
\[
 h(t,x)=\IP_x\big[T_n(\B(1))>t\big]
\]
 (for simplicity in the proof we will omit the 
notation~$\Upsilon_n$ for the projection from~$\R^2$ to~$\R^2_n$,
 starting with 
writing~$\B(1)$ instead of $\Upsilon_n \B(1)$ in this display).
The Brownian motion~$\tX$ on the torus conditioned on not hitting
 the unit ball by time $t_\alpha$ can be defined in an
elementary manner by conditioning by a non-negligible event.
  Consider the time-inhomogeneous diffusion~$\tX$ 
with the transition densities from time~$s$ to time $t>s$:
\begin{equation}
\label{df_tilde_X}
\tilde{p}(s,t,x,y) = \tilde{p}_0(t-s,x,y)
\frac{h(t_\alpha-t,y)}{h(t_\alpha-s,x)},
\end{equation}
where~$\tilde{p}_0$ is the transition density of~$X$
killed on hitting~$\B(1)$. This formula is similar to \eqref{df_hat_p}. 
Denote $\widetilde{T}^{(s)}_n(A)=\inf\{t\geq s: {\widetilde X}_t\in A\}$.
Analogously to \eqref{eq:densityShanghai}, 
we can compute the Radon-Nikodym derivative of the law of $\tX$ on the time-interval $[s, \widetilde{T}^{(s)}_n(\partial \B(n/3))]$ given $\tX_s=x$ with respect to that of $\hW$ 
on $[0, \htau(n/3)]$ started at $x$,
\begin{equation}
\label{RN_torus}
 \frac{d \IP\big[\tX_{[s,\widetilde{T}^{(s)}_n(\partial \B(\frac n3))]}
 \in \cdot \vert \tX_s=x \,\big]}
{d \IP_x\big[\hW_{[0,\htau(\frac n3)]}\in \cdot\,\big]} 
=  \frac{\ln\big(\frac{n}{3\ln n}\big)}{\ln (\frac{n}{3})}
\times \frac{h\big(t_\alpha-\widetilde{T}^{(s)}_n(\partial \B(\frac n3)),
\tX_{\widetilde{T}^{(s)}_n (\partial \B(\frac n3))}\big)}{h(t_\alpha-s,x)},
\end{equation}
for any $x\in \partial \B(\frac{n}{3\ln n})$
(see also~(92) of~\cite{CPV16}).

Next, for a large~$C$
abbreviate $\delta_{n,\alpha} = C\alpha\sqrt{\frac{\ln\ln n}{\ln n}}$ 
and 
\[
 I_{\delta_{n,\alpha}} = \Big[(1-\delta_{n,\alpha})
\frac{2\alpha\ln^2 n}{\ln\ln n},
    (1+\delta_{n,\alpha})\frac{2\alpha\ln^2 n}{\ln\ln n}\Big].
\]
Let~$N_\alpha$ be the number of Brownian motion's excursions 
between $\partial \B\big(\frac{n}{3\ln n}\big)$ 
and $\partial \B(n/3)$ on the torus,
up to time~$t_\alpha$.
It is well known that
 $\IP\big[\B(1)\cap \X_{t_\alpha}^n =\emptyset \big]
 =n^{-2\alpha+o(1)}$
(see e.g.~\cite{BK14}).
Then,
observe that~\eqref{eq_excursions_torus_g} implies that
\begin{align*}
 \IP\big[N_\alpha\notin I_{\delta_{n,\alpha}} \; \big|\;
 \B(1)\cap \X_{t_\alpha}^{(n)}=\emptyset\big] 
\leq \frac{\IP[N_\alpha\notin I_{\delta_{n,\alpha}}]}
{\IP[\B(1)\cap \X_{t_\alpha}^{(n)}=\emptyset]} 
\leq n^{2\alpha+o(1)}
 \times n^{-C'\alpha^2},
\end{align*}
where~$C'$ is a constant that can be made arbitrarily
large by making the constant~$C$ in the definition of~$\delta_{n,\alpha}$
 large enough.
So, if~$C$ is large enough, for some $c''>0$ it holds that
\begin{equation}
\label{cond_numb_exc}
\IP\big[N_\alpha\in I_{\delta_{n,\alpha}} \; 
\big|\; \B(1)\cap \X_{t_\alpha}^{(n)}=\emptyset\big] 
  \geq 1 - n^{-c''\alpha}.
\end{equation}

Now, we estimate the (conditional) probability 
that an excursion hits the set~$A$. For this, observe 
that~\eqref{cond_avoid_A} 
implies that,
for any $x\in \partial \B\big(\frac{n}{3\ln n}\big)$
\begin{equation}
\label{eq:decadix}
 \IP_x\big[\htau_1(A)>\htau_1(\partial \B(n/3))\big]
  = 1 - \frac{\pi}{2}\capa(A)\frac{\ln\ln n}{\ln^2 n}
 \big(1+o(1)\big),
\end{equation}
see also~(84) of~\cite{CPV16}.
This is for $\hW$-excursions, but we also need a corresponding fact 
for $\tX$-excursions.
More precisely, we need to show that
\begin{equation}
\label{tX_hW_excursions}
\IP\big[\widetilde{T}^{(s)}_n(A)
 <\widetilde{T}^{(s)}_n(\partial \B(\frac n3))
 \mid \tX_s=x\big]
=\IP_x\big[\htau(A)<\htau(\partial \B(\frac n3))\big]
\Big(1+O\Big(\frac{1}{\sqrt{\ln n}}\Big)\Big).
\end{equation}

In order to prove the above fact,
we first need the following estimate, which 
(in the discrete setting) was 
proved in~\cite{CPV16} as Lemma~4.2.
For all $\lambda \in (0,1/5)$, there exist $c_1>0, n_1 \geq 2$,
$\sigma_1>0$ (depending on~$\lambda$) such that
for all~$n \geq n_1$, $1\leq \beta\leq \sigma_1\ln n$,  
 $\|x\|,\|y\|\geq \lambda n$, 
$|r|\leq \beta n^2$ and  all~$s\geq 0$, 

\begin{equation}
\label{regularity_h}
\Big|\frac{h(s,x)}{h(s+r,y)}-1\Big| \leq \frac{c_1\beta}{\ln n} \;.
\end{equation}
The proof of the above in the continuous setting is completely analogous.
Then, the idea is to write, similarly to~(86) of~\cite{CPV16}
that
\begin{equation}
 \label{transfer_h}
h(s,x)
= \frac{\ln\big(\frac{n}{3\ln n}\big)}{\ln (\frac{n}{3})} 
\int_{\partial \B(n/3)\times\R_+} h(s-t,y)
\,d\nu(y,t)
 + \psi_{x,s,n} \ ,
\end{equation}
where 
\[
 \nu(M,T)=\IP_x\big[X_{T_n(\partial \B(\frac n3))}\in M,
T_n(\partial \B(\frac n3))\in T  \mid T_n(0)>T_n(\partial \B(\frac n3))\big]
\]
and $\psi_{x,s,n}= P_x[T_n(\partial \B(n/3))\geq T_n(0)>s]$.
Then, one uses~\eqref{transfer_h} together with~\eqref{RN_torus}
to obtain~\eqref{tX_hW_excursions} in a rather standard
way; the only obstacle in adapting the discrete
argument to the continuous setting is that~(87) of~\cite{CPV16}
does not hold for \emph{all} $x\in \R^2_n\setminus\B(1)$
(this is because~$x$ can be very close to~$\B(1)$; that fact 
would be true e.g.\ for all $x\in \R^2_n\setminus\B(2)$,
which would be, unfortunately, not enough to obtain the 
analogue of~(88) of~\cite{CPV16}). 
Note, however, that by a repeated application of~\eqref{regularity_h}
one readily obtains that
\begin{equation}
\label{improvement_CPV}
 \frac{h(s,x)}{h(s+r,y)} \leq \exp\Big(\frac{Cr}{n^2\ln n}\Big)
\end{equation}
(one could use this observation in~\cite{CPV16} as well),
which is even stronger than~(88) of~\cite{CPV16}.

Once we have~\eqref{tX_hW_excursions},
the idea is roughly that 
\begin{align*}
\IP\big[\Upsilon_n A \cap 
   \X_{t_\alpha}^{(n)}=\emptyset \mid \B(1)\cap 
   \X_{t_\alpha}^{(n)}=\emptyset\big]
&{\approx} \E[ (\text{right-hand side of }\eqref{eq:decadix})^{N_\alpha}]\\
&\!\! \stackrel{\eqref{cond_numb_exc}}{\approx} \Big(1 - 
\frac{\pi\ln\ln n}{2\ln^2 n}\capa(A)\Big)^{\frac{2\alpha\ln^2 n}
{\ln\ln n}}
\\
&= \exp\big(-\pi\alpha\capa(A)(1+o(1))\big),
\end{align*}
which would show Theorem~\ref{t_torus}.
Also, one needs to take care of some extra technicalities
(in particular, excursions starting at times close to~$t_\alpha$
need to be treated separately), but the arguments of~\cite[section 4.2]{CPV16}
are quite standard
and  adapt to the continuous case \emph{mutatis mutandis}.
\end{proof}

\subsection{Proof of Theorems~\ref{t_BRI_image},~\ref{t_Phi_process},~\ref{t_Y_process},~\ref{t_Y_^process} and~\ref{t_freedisk}.}

\begin{proof}[Proof of Theorem~\ref{t_BRI_image}]
First, we need the following elementary consequence of
 the Mapping Theorem for Poisson processes 
(e.g.\ Section~2.3 of~\cite{K93}):
if $\mathcal{P}$ is a Poisson process on $\R_+$ with rate
 $r(\rho)=\frac{a}{\rho}$, then the image of~$\mathcal{P}$
under the map $g(\rho) = c \rho^h$ is a Poisson
process~$\mathcal{P}'$
with rate $r'(\rho)=\frac{ah^{-1}}{\rho}$, where~$c$ and~$h$
are positive constants.
Theorem~\ref{t_BRI_image} now follows from
the fact that a conformal image of a Brownian trajectory
is a Brownian trajectory (the fact that we are dealing 
with a conditioned Brownian motion does not change
anything due to Lemma~\ref{l_conn_W_hW}).
\end{proof}

\begin{proof}[Proof of Theorem~\ref{t_freedisk}]
We write
\begin{align*}
 \IP\Big[\frac{2\alpha\ln^2\|x\|}{\ln \Phi_x(\alpha)^{-1}}>s\Big] 
&=
\IP\big[\Phi_x(\alpha) 
> r_s
\big]\\
 &= \IP\big[\B(x,r_s) \subset \V^\alpha \big]\\
 &= \exp\big(-\pi\alpha\capa\big(\B(x,r_s)\cup 
 \B(1)\big)\big),
\end{align*}
 and an application of
Lemma~\ref{l_cap_distantdisk}~(iii)
 concludes the proof of the first part.
For the boundary case $x\in\partial\B(1)$, 
using Lemma~\ref{l_cap_blister} we obtain
\begin{align*}
 \IP\big[\alpha \Phi_x(\alpha)^2>s\big] 
&=  \IP\Big[\Phi_x(\alpha)>\sqrt{\frac{s}{\alpha}}\Big]\\ 
&=  \IP\Big[\B\Big(x,\sqrt{\frac{s}{\alpha}}\Big)\subset\V^\alpha\Big]\\
&=\exp \Big(-\alpha\times \frac{s}{\alpha}
+O\big((s/\alpha)^{3/2}\big)\Big)\\
&= e^{-s}\big(1+O\big((s/\alpha)^{3/2}\big)\big),
\end{align*}
which concludes the proof.
\end{proof}

\begin{proof}[Proof of Theorem~\ref{t_Phi_process}]
Note that the proof of~(ii) follows directly from
the construction, since $\Phi_0(\alpha)=\rho_1^\alpha$.
The fact that the process $(\Phi_x (\alpha) , \alpha>0)$ is Markovian
immediately follows from~\eqref{superp_BRI}.
Let~$G_{\alpha,\delta}$ be the event that there is a jump
in the interval $[\alpha,\alpha+\delta]$.
Then, to compute the jump rate of~$\Phi_x$
given that $\Phi_x(\alpha)=r$, observe that,
by~\eqref{superp_BRI}, \eqref{capa_disk}, and~\eqref{eq_vacant_Bro}
\begin{align*}
 \lim_{\delta\to 0}\frac{\IP\big[
 G_{\alpha,\delta}\big]}{\delta} &=
  \lim_{\delta\to 0}\frac{\IP\big[\Phi_x(\delta) < r\big]}{\delta}\\
  &=  \lim_{\delta\to 0}\frac{1-
  \IP\big[\B(x,r)\subset\V^\delta\big]}{\delta}\\
  &= \lim_{\delta\to 0}\frac{1-\exp\big(-\pi\delta
   \capa \big(\B(1)\cup B(x,r)\big)\big)}{\delta}\\
  &= \pi\capa \big(\B(1)\cup B(x,r)\big).
\end{align*}
Moreover, conditioned on~$G_{\alpha,\delta}$,
we have for 
$s<r$
\[
\IP[V^{(x,r)}<s \mid G_{\alpha,\delta}]
=\frac{1-\IP[\B(s)\subset\V^\delta]}{\IP[G_{\alpha,\delta}]}
=\frac{\pi\delta\capa \big(\B(1)\cup B(x,s)\big)(1+o(1))}
{\pi\delta\capa \big(\B(1)\cup B(x,r)\big)(1+o(1)) },
\]
and we conclude the proof by sending~$\delta$ to~$0$.
\end{proof}

\begin{proof}[Proof of Theorem~\ref{t_Y_^process}] 
For a fixed $x$ denote by  $k_x: \R^+ \to \R^+$  the non-decreasing function  $k_x(r)= \capa\big(\B(1)\cup\B(x,r)\big)$. 
 From Theorem \ref{t_Phi_process} we obtain that
the process $\Phi_x(\alpha)$ is a pure jump Markov process with generator given by
\[
{\cal L}f(r)= \pi \int_0^r [f(s)-f(r)] d k_x(s)\;,
\] 
for $f: [\max\{ 1-\|x\|, 0\}, +\8) \to \R$. 
We define the mappings ${\mathfrak R}_{x,\beta}: \R \to \R$ by
\[
{\mathfrak R}_{x,\beta} (y) =
\begin{cases}
 \exp\big( - e^{-\{y-\beta-\ln(2 \ln^2 \|x\|)\}}\big) ,
& \text{for } \|x\|>1,  \\
 \exp\big(\frac12 (y-\beta) \big) ,
&   \text{for } \|x\|=1     \\
1-\|x\| + \big( \frac{4\sqrt 2}{3 \pi^2} (\frac{1-\|x\|}{\|x\|})^{1/2} e^{y-\beta}\big)^{2/3},
%
%
%
 &  \text{for } \|x\| \in (0,1).  
\end{cases}
\]
Note that ${\mathfrak R}_{x,\beta}$ are
 increasing and one-to-one with that 
${\mathfrak R}_{x,\beta}^{-1}(\Phi_x(e^\beta))= Y_x^{out}(\beta)$, 
$ Y_x^{\6}(\beta)$ or $ Y_x^{in}(\beta)$ 
defined in~\eqref{def:Y_^process},
 according to $x \notin \B(1), x \in \6 \B(1)$ or $0<\|x\|<1$.
Thus, all these three processes are (time inhomogeneous) 
Markov processes, with generators given by 
\[
{\cal L}_{x, \beta}f(y)= f'(y) +\pi \int_{-\8}^y [f(z)-f(y)] e^\beta dk_x( {\mathfrak R}_{x,\beta}(z))
\]
 for smooth $f: \R \to \R$. 
Now, the above particular choices 
of $ {\mathfrak R}_{x,\beta}(\cdot)$ are to ensure that, 
in each of the three different cases,
\[
\lim_{\beta \to \8} \pi e^\beta k_x( {\mathfrak R}_{x,\beta}(y))  = e^y
\]
uniformly on compacts. The convergence follows from 
Lemmas~\ref{l_cap_distantdisk}~(iii), \ref{l_cap_blister},
 and~\ref{l_cap_interior}.
Thus, for the generators themselves we have for fixed~$f$
\begin{equation} \nn
\lim_{\beta \to \8}
{\cal L}_{x, \beta}f(y)={\cal L}f(y)
\end{equation}
uniformly on compacts. With this to hand, we follow the standard 
scheme of compactness and identification of the limit for 
convergence: (i) the family of processes indexed by 
$\beta_w$ is tight in the Skorohod space 
${\mathbb D}(\R_+, \R)$; 
(ii) the limit is solution of the martingale problem associated 
to the generator~$\mathcal L$, which is uniquely determined.
It is not difficult to check that
\begin{equation}\nn
\sup_{\beta \geq 0} e^\beta    k_x( {\mathfrak R}_{x,\beta}(y))  \longrightarrow 0\qquad {\rm as} \; y \to -\8 \;.
\end{equation}
Then we can apply Theorem 3.39 of Chapter IX in~\cite{JS87} 
in order to obtain tightness (the assumptions can be checked 
as in the proof of Theorem 4.8 of Chapter IX in~\cite{JS87} 
which deals with time-homogeneous processes whereas we have here 
a weak inhomogeneity; tightness of the 1-dimensional marginal 
follows from Theorem~\ref{t_freedisk}, that we prove independently, 
for the cases $\|x\|>1$ and $\|x\|=1$; the last case is similar).  
This concludes the proof of 
 convergence of $Y^{out}(\beta_w + \cdot)$, $Y^{\6}(\beta_w + \cdot)$
  and $Y^{in}(\beta_w + \cdot)$ to $Y$ as $\beta_w \to \8$.
\end{proof}

%
\begin{proof}[Proof of Theorem  \ref{t_Y_process}] 
The Markov process $\Phi_0(\alpha)$ described by (ii) in Theorem~\ref{t_Phi_process} is transformed by the change of variables 
$\alpha=e^\beta, y= \ln \ln r + \beta + \ln 2$, or equivalently,
$r=\exp(e^{y-\beta}/2)$, into a Markov process $Y(\beta)$. Indeed, for the new process the jump rate becomes
\[
2 \times \frac{e^{y-\beta}}{2} 
\times \frac{d  \alpha}{d  \beta} = e^y\;,
\]
 and the evolution is given by 
\begin{equation} 
\nn
Y(\beta+ h) = 
\begin{cases}
 Y(\beta) + h,  & \text{with probability }  1-e^yh+o(h),  \\
 Y(\beta) + \ln U +h, &   \text{with probability }  e^yh+o(h) ,
\end{cases}
\end{equation}
where $U$ is an independent Uniform$[0,1]$ random variable. 
Since $-\ln U$ is an Exp(1) variable,
the generator of $Y$ is given by $\cal L$. 

The adjoint generator ${\cal L}^*$ is given by
\[
{\cal L}^* g (y) = - g'(y) + e^y \int_y^{+\8} g(z) dz - e^y g(y) \;.
\]
The negative of a Gumbel variable has density
 $g(y)= \exp( y-e^{y})$,
we easily check that ${\cal L}^* g (y)=0$. 
Hence this law is invariant. 
\end{proof}

\section*{Appendix}

\begin{proof}[Proof of Proposition~\ref{p_prop_gx2}.]
Let $a=\|x\|$. Note that,   
\[
 \frac{1}{2\pi}\int_0^\pi \ln (a^2+1-2a\cos\theta)\,d\theta -\ln a =
   \frac{1}{2\pi}\int_0^\pi \ln (1+a^{-2}-2a^{-1}\cos\theta)\,d\theta,
\]
and that the claim is obviously valid for~$a=0$,
so it remains to prove that
\begin{equation}
\nn
 I(a):= \int_0^\pi \ln (a^2+1-2a\cos\theta)\,d\theta = 0
 \qquad \text{ for all } a\in(0,1].
\end{equation}
By the change of variable $\theta\to \pi-\theta$ we 
find that $I(a)= \int_0^\pi \ln (a^2+1+2a\cos\theta)\,d\theta$,
and so
\begin{align} \label{eq:I(a)}
  I(a) &= \frac{1}{2} \int_0^\pi 
 \ln \big((a^2+1)^2-4a^2\cos^2\theta\big)\,d\theta\\
 &= \int_0^{\pi/2} 
 \ln \big((a^2+1)^2-4a^2\cos^2\theta\big)\,d\theta.
 \nn
\end{align}
Then, using the same trick as above (change the variable
$\theta\to \frac{\pi}{2}-\theta$ so that the cosine becomes sine), 
we find
\begin{align*}
 I(a) &=  \frac{1}{2} \int_0^{\pi/2}
\ln \big[\big((a^2+1)^2-4a^2\cos^2\theta\big)
  \big((a^2+1)^2-4a^2\sin^2\theta\big) \big]\,d\theta\\
 &=  \frac{1}{2} \int_0^{\pi/2} 
\ln \big[(a^2+1)^4 - 4a^2(a^2+1)^2 + 16a^4\cos^2\theta\sin^2\theta \big]
  \,d\theta\\
  &=  \frac{1}{2} \int_0^{\pi/2}
\ln  \big[(a^2+1)^4 - 4a^2(a^2+1)^2 + 4a^4\sin^2 2\theta \big]
  \,d\theta\\
    &=  \frac{1}{2} \int_0^{\pi/2}
\ln  \big[(a^2+1)^4 - 4a^2(a^2+1)^2 + 4a^4 - 
   4a^4\cos^2 2\theta \big]
  \,d\theta\\
      &=  \frac{1}{2} \int_0^{\pi/2}
\ln  \big[\big((a^2+1)^2 - 2a^2\big)^2 - 
   4a^4\cos^2 2\theta \big]
  \,d\theta\\
        &=  \frac{1}{2} \times \frac{1}{2}\int_0^\pi
\ln  \big[(a^4+1)^2  - 
   4a^4\cos^2 \theta \big]
  \,d\theta
\end{align*}
using \eqref{eq:I(a)}, and we finally arrive to the following identity:
\begin{equation} 
\label{magic_identity}
I(a) = \frac{I(a^2)}{2}.
\end{equation}
This implies directly that $I(1)=0$; for $a<1$ 
just iterate~\eqref{magic_identity} and use the obvious fact
that $I(\cdot)$ is continuous at~$0$.
\end{proof}

We have to mention that
other proofs are available as well; see~\cite[Ch.~20]{Chen10}.

\section*{Acknowledgements}
The authors thank Christophe Sabot for helping with the
rigorous definition of the process~$\RR$ starting from~$\RR_0=1$,
and Alexandre Eremenko for helping with 
the proof of Lemma~\ref{l_cap_blister}.
The work of S.P.\ was partially supported by
CNPq (grant 300886/2008--0) and FAPESP (grant 2017/02022--2). 
The work of F.C.\ was partially supported by CNRS (LPSM, UMR 8001). 
Both of us have beneficiated from support of Math Amsud programs 15MATH01-LSBS and 19MATH05-RSPSM.

\small{

}

\end{document}